\documentclass[12pt]{article}
\usepackage{amsmath}
\usepackage{amsfonts}
\usepackage{amssymb}
\usepackage{amscd}
\usepackage{amsthm}
\usepackage{bbm}
\usepackage{color}      
\usepackage[linktocpage=true,plainpages=false,pdfpagelabels=false]{hyperref}

\input xypic

\usepackage[matrix,arrow,curve]{xy}

\newcommand{\Z}{{\mathbb Z}}
\newcommand{\Hb}{{\mathbb H}}

\newcommand{\N}{{\mathbb N}}
\newcommand{\Q}{{\mathbb Q}}

\newcommand{\C}{{\mathbb C}}

\newcommand{\Hc}{{\mathcal H}}

\newcommand{\res}{\mathop{\rm res}}       

\newcommand{\LL}{\mathcal L}

\newcommand{\Ab}{{\mathbb A}}
\renewcommand{\a}{\mathfrak{a}}

\newcommand{\p}{\mathfrak{p}}

\newcommand{\OO}{{\mathcal O}}

\newcommand{\Map}{{\rm Map}}

\newcommand{\Hom}{{\rm Hom}}

\newcommand{\Ker}{{\rm Ker}}

\newcommand{\Spec}{{\rm Spec} }

\newcommand{\sgn}{{\rm sgn}}

\newcommand{\M}{{\rm M}}
\newcommand{\Spf}{{\rm Spf} }

\newcommand{\gm}{\mathbb G_{m}}
\newcommand{\ga}{\mathbb G_{a}}

\newcommand{\z}{{\mathbb{Z}}}

\newcommand{\lrto}{\longrightarrow}

\newcommand{\uz}{\underline{\mathbb{Z}}}

\theoremstyle{plain}
\newtheorem{theor}{Theorem}[section]
\newtheorem{prop}[theor]{Proposition}

\newtheorem{corol}[theor]{Corollary}
\newtheorem{lemma}[theor]{Lemma}

\theoremstyle{remark}
\newtheorem{rmk}[theor]{Remark}

\theoremstyle{definition}
\newtheorem{defin}[theor]{Definition}
\newtheorem{defin-prop}[theor]{Definition-Proposition}

\numberwithin{equation}{section}

\newcommand{\supp}{{\rm supp}}

\newcommand{\quash}[1]{}

\title{Continuous homomorphisms between algebras of iterated Laurent series over a ring\footnotetext{This work is supported by the RSF under a grant 14-50-00005.}}
\author{Sergey Gorchinskiy and Denis Osipov}
\date{}

\begin{document}
\maketitle

\begin{abstract}
We study continuous homomorphisms between algebras of iterated Laurent series over a commutative ring. We give a full description of such homomorphisms in terms of a discrete data determined by the images of parameters. In similar terms, we give a criterion of invertibility of an endomorphism and provide an explicit formula for the inverse endomorphism. We study the behavior of the higher-dimensional residue under continuous homomorphisms.
\end{abstract}

\tableofcontents

\section{Introduction}

In this paper, we study continuous homomorphisms between algebras of iterated Laurent series in many variables over a commutative ring. Let us start with a previously known case of one variable. Let $A$ be a commutative ring and let $A((t))$ be the ring of Laurent series over $A$. For simplicity, assume in the introduction that $A$ does not decompose into a product of two rings. Then for any invertible Laurent series
$$
\varphi =\mbox{$ \sum\limits_{l \in \z} a_l t^l $} \in A((t))^*\,,
$$
there is an integer $\nu(\varphi) \in \z$ such that the coefficient $a_{\nu(\varphi)}$ is invertible and for all~${l<\nu(\varphi)}$, the coefficients $a_l$ are nilpotent (see more explanation and references in Section~\ref{prelim}).

If $\nu(\varphi) > 0$, then we have a well-defined continuous endomorphism of the \mbox{$A$-algebra}
\begin{equation}\label{map}
A((t))\longrightarrow A((t))\,,\qquad  \mbox{$f=\sum\limits_{l \in \Z} b_lt^l \longmapsto f(\varphi)=\sum\limits_{l \in \Z}  b_l \varphi^l  $}\,,
\end{equation}
where we consider the natural topology on~$A((t))$ with the base of open neighborhoods of zero given by $A$-submodules $t^iA[[t]]$,~$i\in\Z$. The series $\sum\limits_{l \in \Z}  b_l \varphi^l$ converges in this topology. It is not hard to check that all continuous endomorphisms of the \mbox{$A$-algebra}~$A((t))$ have this form. Besides, Morava (after Kapranov) in~\cite[\S\,1.2]{Mo} and Mu\~{n}oz Porras and Plaza Mart\'in in~\cite[Theor.\,3.3]{MPM} have independently obtained a non-trivial statement that the endomorphism in formula~\eqref{map} is invertible if and only if~$\nu(\varphi)=1$.

Note that the nilpotent coefficients $a_l$ of $\varphi$ with $l < \nu(\varphi)$ play an important role. Even in the case of continuous automorphisms of the $A$-algebra $A[[t]]$, where we have $\nu(\varphi)=1$ and $a_l=0 $ for $l <0$, one needs to consider the nilpotent coefficient $a_0$ in order to obtain the algebra of all derivations (or vector fields) on the formal disk~$\Spf\big(\C[[t]]\big)$ as the Lie algebra of the group of automorphisms, see~\cite[Ch.\,6.2]{FB}.

\medskip

We generalize the above facts to the case of continuous homomorphisms~$A((t_1))\ldots((t_n))\to A((t_1))\ldots((t_m))$ between \mbox{$A$-algebras} of iterated Laurent series over~$A$ with the natural topology.

As a geometric motivation, let us mention that when $A=k$ is a field, the \mbox{$n$-dimensional} local field $k((t_1))\ldots ((t_n))$ appears naturally from iterated localization and completion procedures along a full flag of irreducible subvarieties on an $n$-dimensional algebraic variety~$X$, see, e.g., a survey~\cite{O}.
Thus, in particular, for a finite-dimensional $k$-algebra $A$, the ring
$$
A((t_1))  \ldots ((t_n))\simeq A\otimes_k k((t_1))  \ldots ((t_n))
$$
can be considered as a deformation of the $n$-dimensional local field along $A$, or as a ring that appears from the scheme $X \times_k A$ over $A$.

\medskip

Let us describe our main results. For short, we denote $A((t_1))\ldots((t_n))$ by $\LL^n(A)$. Elements of $\LL^n(A)$ have the form ${f=\sum\limits_{l\in\z^n}a_lt_1^{l_1}\ldots t_n^{l_n}}$, where ${l=(l_1,\ldots, l_n)\in\z^n}$ and~${a_l\in A}$, with a certain restriction on the set of indices of non-zero coefficients, see Section~\ref{prelim}.

Firstly, we give the following description of all continuous homomorphisms between \mbox{$A$-algebras} of iterated Laurent series, see Theorem~\ref{prop:contchange}. Remind that the classical valuation (when $A$ is a field) is generalized to a homomorphism of groups
$$
\nu\,:\,\LL^n(A)^*\lrto\z^n\,.
$$
For instance, we have $\nu(t_1)=(1,0,\ldots,0)$ and $\nu(t_n)=(0,\ldots,0,1)$.
Given a collection of~$n$ invertible iterated Laurent series $\varphi_1,\ldots,\varphi_n\in \LL^m(A)^*$ in~$m$ variables, suppose that the \mbox{$(m\times n)$-matrix} $\big(\nu(\varphi_1),\ldots,\nu(\varphi_n)\big)$ with integral entries is in column echelon form with positive leading entries (see Lemma~\ref{lem:upper}$(ii)$). For example, when $m=n$, this means that the matrix is upper-trangular with positive diagonal entries. Then we have a well-defined continuous homomorphism of $A$-algebras
$$
\phi\,:\,\LL^n(A)  \lrto \LL^m(A)\,,
$$
$$
\mbox{$f=\sum\limits_{l\in\z^n}a_lt_1^{l_1}\ldots t_n^{l_n}\longmapsto f(\varphi_1,\ldots,\varphi_n)=\sum\limits_{l\in\z^n}a_l\varphi_1^{l_1}\ldots\varphi_n^{l_n}$}\,.
$$
Moreover, all continuous homomorphisms of $A$-algebras $\LL^n(A)\to\LL^m(A)$ have this form.

In particular, this implies that there are no such homomorphisms when $n>m$, see Corollary~\ref{cor:inequality}.

As another application, we obtain that the functor that sends a commutative ring~$A$ to the set of all continuous homomorphisms of $A$-algebras $\LL^n(A)\to\LL^m(A)$ is represented by an ind-affine scheme over $\z$ that has many nice geometric properties, see Proposition~\ref{prop:repr} and Corollary~\ref{cor:repr}. More precisely, this ind-scheme is a product of an ind-flat scheme over~$\z$ with a so-called thick ind-cone introduced and studied in~\cite{GOMS}. Such ind-schemes provide an adequate replacement of ind-flat schemes in the context of iterated Laurent series in many variables.

\medskip

Secondly, we investigate how the residue of top differential forms is changed under continuous endomorphisms. More precisely, let the free rank one $\LL^n(A)$-module ${\widetilde{\Omega}^n_{\LL^n(A)}\simeq \LL^n(A)dt_1\wedge\ldots\wedge dt_n}$ be the natural quotient of the $\LL^n(A)$-module of absolute K\"ahler $n$-differentials of the ring $\LL^n(A)$. Further, let
$$
\res\,:\,\widetilde{\Omega}^n_{\LL^n(A)}\longrightarrow A\,,\qquad \mbox{$\sum\limits_{l\in\z^n}a_lt_1^{l_1}\ldots t_n^{l_n}dt_1\wedge\ldots\wedge dt_n\longmapsto a_{-1\ldots -1}$}\,,
$$
be the $n$-dimensional residue map. In Proposition~\ref{prop:invres} and Corollary~\ref{cor:res}, we describe explicitly how the residue map is changed under a continuous homomorphism~$\phi$ as above. In particular, in the case $m=n$, we show that the residue map is invariant under all continuous automorphisms of the $A$-algebra $\LL^n(A)$, see Remark~\ref{rmk:invres}.

As an application, we obtain injectivity of the homomorphisms under a rather mild assumption, see Corollary~\ref{cor:inj} and Remark~\ref{rmk:inj}.

Note that when $A$ is a field, the invariance of the residue under continuous automorphisms of the field $\LL^n(A)$ over $A$  is classical in the case $n=1$ (see, e.g.,~\cite{S}) and is known in the case $n\geqslant 2$ due to the works of Parshin, Lomadze, and Yekutieli, see~\cite[\S\,1, Prop.\,1]{P0},~\cite[Lem.\,6(VIII)]{Lom}, and~\cite[Theor.\,2.4.3]{Y}. Also, let us mention that unlike the case $n=1$, when $n\geqslant 2$, the residue is not necessarily invariant under a non-continuous automorphism, see~\cite[Ex.\,2.4.24]{Y} (when $n=1$, any automorphism of the field $A((t))$ is continuous, see, e.g.,~\cite[Ch.\,II, \S\,2, Exer.\,6\,a)]{FV}).

\medskip

Thirdly, when $m=n$, we give a criterion of invertibility of a continuous endomorphism $\phi$, see Theorem~\ref{theor:inv}. Namely,~$\phi$ is invertible if and only if the \mbox{$(n\times n)$-matrix} ${\big(\nu(\varphi_1),\ldots,\nu(\varphi_n)\big)}$ is invertible, or, equivalently, this upper-triangular matrix has units on the diagonal. Moreover, it turns out that if $\phi$ is invertible, then the inverse $\phi^{-1}$ coincides with the adjoint map $\phi^{\vee}$ to the induced map $\phi\colon\widetilde{\Omega}^n_{\LL^n(A)}\to \widetilde{\Omega}^n_{\LL^n(A)}$ (denoted also by~$\phi$) with respect to the perfect continuous pairing
$$
\LL^n(A)  \times \widetilde{\Omega}^n_{\LL^n(A)}  \lrto  A\,,\qquad (f, \omega) \longmapsto \res(f \omega)\,.
$$

The generalization of the invertibility criterion from $n=1$ to $n\geqslant 1$ is non-trivial. It is not clear how to apply methods from~\cite{Mo} and~\cite{MPM} in the case $n>1$ for an arbitrary commutative ring $A$ (cf.~\cite[Rem.\,5.2]{OZ1} for $n=2$). Our proof of the criterion is based on a different method: we use the invariance of the residue map and the above representability result together with the theory of thick ind-cones from~\cite{GOMS}.

\medskip

Finally, the description of the inverse of a continuous endomorphism $\phi$ as the adjoint map gives immediately an explicit formula for~$\phi^{-1}=\phi^{\vee}$, see in Remark~\ref{rmk:expladj} an explicit formula for $\phi^{\vee}$. For simplicity, let us give this formula when $n=1$. Even in this case, the formula seems not to be known before. If $\phi$ is an endomorphism as in formula~\eqref{map} given by an element $\varphi\in A((t))^*$ with $\nu(\varphi)=1$, then for any element $f \in A((t))$, we have
$$
\mbox{$\phi^{-1}(f)=\phi^{\vee}(f)= \sum\limits_{l\in\z}\res\big(f\varphi^{-l-1}  (\partial \varphi/ \partial t)dt \big)t^l$}\,.
$$
This implies the following peculiar identity (for the general case $n\geqslant 1$, see formula~\eqref{eq:identity}):
$$
f=\mbox{$\sum\limits_{l\in\z}\res\big(f\varphi^{-l-1}   (\partial \varphi/ \partial t)  dt \big)\varphi^l$}\,.
$$
We do not know any elementary (direct computational) proof of this formula even when~$A$ is a field.

\medskip

The results of this paper will be applied in a forthcoming paper~\cite{GOnext} to further investigations of the $n$-dimensional Contou-Carr\`ere symbol, which was studied  by Contou-Carr\`{e}re himself in~\cite{CC1, CC2} for the case~$n=1$, by the second named author and Zhu in~\cite{OZ1} for the case $n=2$, and by the authors in~\cite{GO1, GOMilnor, GOMS} for  arbitrary~$n$.

\medskip

We are grateful to the referee for his comments.

\section{Preliminaries and notation}   \label{prelim}

In this section, we introduce notation and recall, mainly from~\cite{GOMS}, facts on the ring of iterated Laurent series, the topology on this ring, the group of its invertible elements, its differential forms, and on the higher-dimensional residue map.

\medskip

For short, by a ring, we mean a commutative associative unital ring and similarly for algebras.

Throughout the paper, $A$ denotes a ring. Let $n\geqslant 1$ be a positive integer. Let ${\LL(A):=A((t))=A[[t]][t^{-1}]}$ be the ring of Laurent series over $A$ and let
$$
\LL^n(A):=\LL\big(\LL^{n-1}(A)\big)=A((t_1))\ldots((t_n))
$$
be the ring of iterated Laurent series over $A$. We call elements $t_1,\ldots,t_n$ variables or parameters. Put $t^{l}:=t_1^{l_1}\ldots t_n^{l_n}$ for an element ${l=(l_1,\ldots,l_n)\in\z^n}$. More generally, for any collection $\varphi=(\varphi_1,\ldots,\varphi_n)$, where $\varphi_i\in\LL^n(A)$, and for an element ${l=(l_1,\ldots,l_n)\in\z^n}$, we put $\varphi^l:=\varphi_1^{l_1}\ldots\varphi_n^{l_n}$. Given a formal series $f=\sum\limits_{l\in \z^n} a_l t^{l}$, let the support $\supp(f)\subset\z^n$ of~$f$ be the set of all $l\in\z^n$ such that $a_l\ne 0$. Elements of $\LL^n(A)$ are described explicitly in terms of their supports as follows.

\medskip

Define iteratively the sets
$$
\Lambda_1:=\z\,,\qquad \Lambda_n:=\Map(\z,\Lambda_{n-1})\times \z\,,
$$
where $\Map(\z,\Lambda_{n-1})$ denotes the set of all maps from $\z$ to $\Lambda_{n-1}$. The set $\Lambda_n$ has a natural group structure defined inductively by sums of integers and point-wise sums of group valued functions. For an element $\lambda\in\z$, define the set
$$
\z_{\lambda}=\z^1_{\lambda}:=\{l\in \z\,\mid\,l\geqslant \lambda\}\subset \z\,.
$$
For an element $\lambda=(\lambda',\lambda_n)\in\Lambda_n$ with $n\geqslant 2$, define inductively the set
$$
{\z^n_{\lambda}:=\bigcup\limits_{l_n\geqslant\lambda_n}\z^{n-1}_{\lambda'(l_n)}\times\{l_n\}}\subset\z^n\,.
$$
Explicitly, we have
$$
\Lambda_n=\big\{(\lambda_1,\ldots,\lambda_n)\;\mid\;\lambda_p \, \colon  \, \z^{n-p}\lrto \z \quad \mbox{for} \quad 1\leqslant p\leqslant n-1,\quad \lambda_n\in\Z \big\}
$$
and for an element $\lambda=(\lambda_1,\ldots,\lambda_n)\in\Lambda_n$, we have
$$
\z^n_{\lambda}=\big\{(l_1,\ldots,l_n)\in\Z^n\;\mid\;l_n\geqslant \lambda_n,\,l_{n-1}\geqslant \lambda_{n-1}(l_n),\,\ldots,\,l_{1}\geqslant \lambda_{1}(l_2,\ldots,l_n)\big\}\,.
$$
Given two subsets $X,Y\subset \z^n$, let $X+Y$ be the subset of $\z^n$ that consists of all pair-wise sums $x+y\in\z^n$, where $x\in X$ and $y\in Y$. Similarly, let the subset $-X\subset\z^n$ consist of all elements $-x$, where $x\in X$. For short, we put $0:=(0,\ldots,0)$.

The ring of iterated Laurent series~$\LL^n(A)$ consists of all series $f=\sum\limits_{l\in \z^n} a_l t^{l}$ such that $\supp(f)\subset \z^n_{\lambda}$ for some $\lambda=(\lambda',\lambda_n)\in\Lambda_n$. Explicitly, we have that $f=\sum\limits_{i\geqslant{\lambda_n}}g_i t_n^i$, where $g_i\in \LL^{n-1}(A)$, $i\geqslant \lambda_n$, are such that $\supp(g_i)\subset\z^{n-1}_{\lambda'(i)}$.

Given two iterated Laurent series $f=\sum\limits_{l\in\z^n_{\lambda}}a_lt^l$ and $g=\sum\limits_{l\in\z^n_{\mu}}b_lt^l$, we have that ${fg=\sum\limits_{l\in\z^n}c_lt^l}$, where
$c_l=\sum\limits_{p+q=l} a_p b_q$. Since $fg$ is a well-defined iterated Laurent series, we obtain that for all $\lambda,\mu\in\Lambda_n$, the summation map ${\z^n_{\lambda}\times\z^n_\mu\to \z^n}$ has finite fibers and there is $\rho\in\Lambda_n$ such that ${\z^n_{\lambda}+\z^n_{\mu}\subset \z^n_{\rho}}$.

\medskip

One has a natural topology on $\LL^n(A)$ such that $\LL^n(A)$ is a topological group with the group structure given by sums of iterated Laurent series. The topology was introduced first by Parshin in~\cite{P1} in the case when $A$ is a finite field. For properties of this topology in the general case, see~\cite[\S\,3.2]{GOMS}. The base of open neighborhoods of zero is given by $A$-submodules $U_{\lambda}\subset\LL^n(A)$, where $\lambda\in\Lambda$ and $U_{\lambda}$ consists of all elements $f\in\LL^n(A)$ such that $\supp(f)\cap (-\z^n_{\lambda})=\varnothing$. Note that this definition of the base of topology looks differently than the one in~\cite{GOMS} but two definitions are evidently equivalent. A countable set $\{f_i\}$, $i\in\N$, of iterated Laurent series $f_i\in\LL^n(A)$ tends to zero if for any $\lambda\in\Lambda_n$, all but finitely many $i\in\N$ satisfy $\supp(f_i)\cap(-\z^n_{\lambda})=\varnothing$, or, equivalently, $0\notin\supp(f_i)+\z^n_{\lambda}$.

The topological group $\LL^n(A)$ is complete, see~\cite[Rem.\,3.4]{GOMS}, and a series $\sum\limits_{i\in\N}f_i$ of elements of $\LL^n(A)$ converges if and only if the countable set $\{f_i\}$, $i\in\N$, tends to zero, see~\cite[\S\,3.3]{GOMS} (if this holds, the result of the summation does not depend on the order of summation).

Note that when $n\geqslant 2$, the product of Laurent series is not continuous with respect to this topology, whence~$\LL^n(A)$ is not a topological ring. Nevertheless, the product with a fixed element is a continuous homomorphism from $\LL^n(A)$ to itself.

\medskip

Define inductively the lexicographical order on $\z^n$ as follows: we have $(l_1,\ldots,l_n)\leqslant (l'_1,\ldots,l'_n)$ if and only if either $l_n < l'_n$, or $l_n=l'_n$ and $(l_1, \ldots, l_{n-1}) \leqslant (l'_1, \ldots, l'_{n-1})$. Clearly, the order is invariant under translations on the group $\z^n$.

Let~$\uz$ denote the constant sheaf in Zariski topology associated with the constant presheaf~$\z$ on~$\Spec(A)$. Thus $\uz(A)$ consists of all locally constant functions on $\Spec(A)$ with values in $\z$.

Let $\LL(A)^*$ be the group of invertible elements in the ring $\LL(A)$. By~\cite[\S\,4.2]{GOMS}, for any element $f\in\LL^n(A)^*$, there is a finite decomposition into a product of rings ${A\simeq\prod\limits_{i=1}^N A_i}$ with the following property. Let $f=\prod\limits_{i=1}^N f_i$ with $f_i=\sum\limits_{l\in\z^n}a_{i,l}t^l$, $a_{i,l}\in A_i$, be the decomposition of $f$ with respect to the arising decomposition $\LL^n(A)\simeq\prod\limits_{i=1}^N\LL^n(A_i)$. Then for for each $i$, $1\leqslant i\leqslant N$, there is an element $l_i\in\z^n$ such that for any $l<l_i$, the element $a_{i,l}\in A_i$ is nilpotent and the element $a_{i,l_i}\in A_i$ is invertible. Moreover, the iterated Laurent series ${\sum\limits_{l<l_i}a_{i,l}t^l\in\LL^n(A_i)}$ is nilpotent. Let $\underline{l}\in\uz^n(A)$ be the locally constant function on $\Spec(A)$ whose value on $\Spec(A_i)$ is $l_i$ for any $i$, $1\leqslant i\leqslant N$. Then the map
$$
\nu\,:\,\LL^n(A)^*\longrightarrow \uz^n(A)\,,\qquad f\longmapsto \underline{l}\,,
$$
is a homomorphism of groups. These facts were first proved by Contou-Carr\`ere in the case $n=1$, see~\cite{CC1} and~\cite{CC2}, and by the second named author and Zhu in the case $n=2$, see~\cite{OZ1}.

\medskip

Given a functor $F$ on the category of rings, by $F_A$ denote the restriction of $F$ to the category of $A$-algebras. If a functor $F$ is represented by an (ind-)scheme, then we denote this (ind-)scheme also by $F$. By a subgroup $G\subset F$, we mean a morphism of functors $G\to F$ such that for any ring $A$, the map $G(A)\to F(A)$ is injective.

By~$L^n\ga$ and $L^n\gm$ denote the group functors on the category of rings $A\mapsto \LL^n(A)$ and $A\mapsto \LL^n(A)^*$, respectively. These functors are represented by ind-affine schemes that have many nice geometric properties, see~\cite[\S\S\,6.2, 6.3]{GOMS}.

Explicitly, $L^n\ga$ is represented by the ind-closed subscheme
$$
\mbox{``$\varinjlim\limits_{\lambda\in \Lambda_{n}}$''}\,\Ab^{\z^n_{\lambda}}\subset\Ab^{\z^n}\,,
$$
which is an ind-affine space, where for a set $E$, by $\Ab^E$ we denote the affine space whose set of coordinates is bijective with $E$. An iterated Laurent series $\sum\limits_{l\in \z^n}a_lt^l\in\LL^n(A)$ corresponds to the point~$(a_l)\in\Ab^{\z^n}(A)$, where $l\in \z^n$.

Let $(L^n\gm)^0$ be the kernel of the morphism of group functors $\nu\colon L^n\gm\to\uz^n$. The group functor $(L^n\gm)^0$ is represented by an absolutely connected ind-affine scheme over~$\z$, see~\cite[Def.\,5.18]{GOMS} and~\cite[Prop.\,6.13(iv)]{GOMS}.

\medskip

By $\widetilde{\Omega}^1_{\LL^n(A)}$ denote the quotient of the $\LL^n(A)$-module of absolute K\"ahler differentials of the ring $\LL^n(A)$ by the $\LL^n(A)$-submodule generated by all elements of the form ${df-\sum\limits_{i=1}^n\frac{\partial f}{\partial t_i}}dt_i$. Then $\widetilde{\Omega}^1_{\LL^n(A)}$ is a free $\LL^n(A)$-module of rank $n$. For each $i\geqslant 0$, put $\widetilde{\Omega}^i_{\LL^n(A)}:=\bigwedge^i_{\LL^n(A)}\widetilde{\Omega}^1_{\LL^n(A)}$. Note that one has a well-defined de Rham differential $d\colon \widetilde{\Omega}^i_{\LL^n(A)}\to \widetilde{\Omega}^{i+1}_{\LL^n(A)}$, $i\geqslant 0$.

We have an $A$-linear {\it residue map}
$$
\res\,:\,\widetilde{\Omega}^n_{\LL^n(A)}\longrightarrow A\,,\qquad \mbox{$\sum\limits_{l\in\z^n}a_lt^l\cdot dt_1\wedge\ldots\wedge dt_n\longmapsto a_{-1\ldots-1}$}\,.
$$
If $A$ is $\Q$-algebra, then the residue map induces an isomorphism (see, e.g.,~\cite[Lem.\,3.12]{GOMS})
$$
\widetilde{\Omega}^n_{\LL^n(A)}/d\widetilde{\Omega}^{n-1}_{\LL^n(A)}\stackrel{\sim}\longrightarrow A\,.
$$

\section{Continuous and functorial additive homomorphisms}

We compare in this section continuous additive homomorphisms $\LL^n(A)\to\LL^m(A)$ with morphisms of groups functors $L^n\ga\to L^m\ga$, see Proposition~\ref{prop:funccont} and Corollary~\ref{cor:funccontrg}. As an application, we describe in Proposition~\ref{prop:valuation} how the homomorphism ${\nu\colon\LL^n(A)^*\to\uz(A)^n}$ is changed under continuous homomorphisms of $A$-algebras of iterated Laurent series.

\medskip

We start with the following simple facts.

\begin{lemma}\label{cor:conv}
\hspace{0cm}
\begin{itemize}
\item[(i)]
For any element $f=\sum\limits_{l\in\z^n}a_lt^l$ in $\LL^n(A)$, the series $\sum\limits_{l\in\z^n}a_lt^l$ converges to~$f$ in the topological group $\LL^n(A)$.
\item[(ii)]
The set of all Laurent polynomials is dense in $\LL^n(A)$.
\end{itemize}
\end{lemma}
\begin{proof}
$(i)$ Consider an element $\lambda\in\Lambda_n$ and the corresponding open neighborhood of zero $U_{\lambda}\subset\LL^n(A)$. Since the summation map ${\supp(f)\times\z^n_{\lambda}\to \z^n}$ has finite fibers, there is a finite subset $S\subset\supp(f)$ such that $0\notin \big(\supp(f)\smallsetminus S\big)+\z^n_{\lambda}$.
Hence for any finite subset $T\subset \supp(f)$ such that $S\subset T$, the partial sum $f_T:=\sum\limits_{l\in T}a_lt^l$ satisfies the condition $f-f_T\in U_{\lambda}$. Therefore partial sums of the series $\sum\limits_{l\in\z^n}a_lt^l$ tend to~$f$ (with respect to any linear order on the summands of the series).

$(ii)$ This follows directly from item~$(i)$.
\end{proof}

\begin{lemma}\label{lemma:intersect}
For any subset $X\subset \z^n$, the following conditions are equivalent:
\begin{itemize}
\item[(i)]
for some $\lambda\in\Lambda_n$, we have $X\subset \z^n_{\lambda}$;
\item[(ii)]
for any $\mu\in\Lambda_n$, the intersection $X\cap (-\z^n_{\mu})$ is finite.
\end{itemize}
\end{lemma}
\begin{proof}
Since the summation map $\z^n_{\lambda}\times \z^n_{\mu}\to\z^n$ has finite fibers, $(i)$ implies $(ii)$. Let us show by induction on $n$ that $(ii)$ implies $(i)$.
The base $n=1$ is obvious and let us make the induction step from $n-1$ to $n$, where $n\geqslant 2$. Note that there are equalities
$$
\z^n=\bigcup\limits_{\lambda\in\Lambda_n}\z^n_{\lambda}=\bigcup\limits_{\lambda\in\Lambda_n}(-\z^n_{\lambda})\,.
$$
Using the analogous fact with $n-1$, we see that there is $\lambda_n\in\z$ such that for any $l_n<\lambda_n$, we have $X\cap \big(\z^{n-1}\times\{l_n\}\big)=\varnothing$. Moreover, for all $l_n\geqslant \lambda_n$ and $\rho\in\Lambda_{n-1}$, the intersection $X\cap \big( (-\z^{n-1}_{\rho})\times\{l_n\}\big)$ is finite.
Thus we conclude by the induction hypothesis applied to the intersections $X\cap\big(\z^{n-1}\times\{l_n\}\big)$ with $l_n\geqslant \lambda_n$.
\end{proof}

\medskip

Let $\ga$ be the additive group scheme $\Spec\big(\z[T]\big)$ over $\z$. Define the ind-affine group scheme $\ga^{\oplus\N}:=\mbox{``$\varinjlim\limits_{i\in \N}$''}(\ga)^i$. Note that the group functor~$\ga^{\oplus \N}$ has a natural $\ga$-module structure in the following sense: for any ring $B$, the group $(\ga^{\oplus\N})(B)\simeq B^{\oplus\N}$ has a canonical structure of a $B$-module, which is functorial with respect to~$B$. Similarly, the group functor $L^n\ga$ has also a natural $\ga$-module structure. We say that a morphism of group functors $L^n\ga\to\ga^{\oplus\N}$ is {\it $\ga$-linear} if it respects the corresponding $B$-module structures for any ring $B$. One defines similarly $(\ga)_A$-linear morphisms of group functors $(L^n\ga)_A\to(\ga^{\oplus\N})_A$.

Given a $(\ga)_A$-linear morphism of group functors
$$
\pi\,:\,(L^n\ga)_A\longrightarrow (\ga^{\oplus\N})_A\,,
$$
we have the kernel $\Ker\big(\pi(A)\big)\subset \LL^n(A)$ of the evaluation $\pi(A)$ of $\pi$ at $A$
$$
\pi(A)\,:\, (L^n\ga)_A(A)=\LL^n(A)\longrightarrow (\ga^{\oplus\N})_A(A)=A^{\oplus \N}\,.
$$
Clearly, $\Ker\big(\pi(A)\big)$ is an $A$-submodule of $\LL^n(A)$. The following statement provides an important relation between the ind-affine scheme~$(L^n\ga)_A$ and the topology on $\LL^n(A)$.

\begin{lemma}\label{lemma:basechar}
The collection of the subgroups $\Ker\big((\pi(A)\big)\subset \LL^n(A)$, where ${\pi\colon (L^n\ga)_A\to (\ga^{\oplus\N})_A}$ runs over all $(\ga)_A$-linear morphisms of group functors, form a base of open neighborhoods of zero in the topological group $\LL^n(A)$.
\end{lemma}
\begin{proof}
Let us show that for any $\pi$ as in the lemma, the kernel $\Ker\big((\pi(A)\big)$ is an open subgroup of $\LL^n(A)$. For any set $E$, a $(\ga)_A$-linear morphism of group schemes ${(\Ab^E)_A\to (\ga)_A}$ over~$A$ is given by a linear functional $(x_e)\mapsto\sum\limits_{e\in E} c_{e} x_e$, where $c_e\in A$ are scalars, $x_e$ are coordinates on $\Ab^E$, and all but finitely many $c_e$ equal zero.

Since the functor~$L^n\ga$ is represented by a formal direct limit of group schemes~$\Ab^{\z^n_{\mu}}$, where $\mu\in\Lambda_n$, we see that a $(\ga)_A$-linear morphism of group functors ${\chi\colon (L^n\ga)_A\to (\ga)_A}$ is given by a linear functional $\sum\limits_{l\in\z^n}a_lt^l\mapsto\sum\limits_{l\in \z^n} c_{l} a_l$, where $c_l\in A$ and for any $\mu\in\Lambda_n$, the intersection $\supp(\chi)\cap \z^n_{\mu}$ is finite. Here, $\supp(\chi)\in \z^n$ is the set of all $l\in \z^n$ such that $c_l\ne 0$.

It follows that a $(\ga)_A$-linear morphism of group functors $\pi\colon (L^n\ga)_A\to (\ga^{\oplus\N})_A$ is given by a collection of $(\ga)_A$-linear morphisms of group functors $(\chi_1,\ldots,\chi_i,\ldots)$, where $\chi_i\colon (L^n\ga)_A\to (\ga)_A$, $i\geqslant 1$. Note that not any collection $(\chi_1,\ldots,\chi_i,\ldots)$ corresponds to a morphism of group functors $\pi$ as above. In particular, for any $\mu\in\Lambda_n$, the intersection of the set $X:=\bigcup\limits_{i\in \N}\supp(\chi_i)$ with $\z^n_{\mu}$ should be finite. This is because the restrictions of all but finitely many~$\chi_i$ to~$\Ab^{\z^n_{\mu}}$ are equal to zero. By Lemma~\ref{lemma:intersect}, we have that $X\subset(-\z^n_{\lambda})$ for some $\lambda\in\Lambda_n$. Therefore the kernel $\Ker\big((\pi(A)\big)$ contains the open subgroup $U_{\lambda}\subset\LL^n(A)$, whence $\Ker\big((\pi(A)\big)$ is open.

Finally, take an open subset $U_{\lambda}$, where $\lambda\in\Lambda_n$, and choose any bijection between~$-\z^n_{\lambda}$ and~$\N$. Then the collection of linear functionals $\sum\limits_{l\in\z^n}a_lt^l\mapsto a_l$, where $l\in-\z^n_{\lambda}$, defines a morphism of group functors $\pi\colon (L^n\ga)_A\to (\ga^{\oplus\N})_A$ such that $\Ker\big((\pi(A)\big)=U_{\lambda}$. This proves the lemma.
\end{proof}


\medskip

Here is our main object of study.

\begin{defin}\label{def:hom}
Denote by
$$
{\Hom^{\rm c}_A\big(\LL^n(A),\LL^m(A)\big)}
$$
the set of all continuous $A$-linear homomorphisms of additive topological groups~${\LL^n(A)\to\LL^m(A)}$. Denote by
$$
{\Hom^{\rm c,alg}_{A}\big(\LL^n(A),\LL^m(A)\big)}
$$
the set of all continuous homomorphisms of $A$-algebras $\LL^n(A)\to\LL^m(A)$.
\end{defin}

Also, let ${\Hom\big((L^n\ga)_A,(L^m\ga)_A\big)}$ be the set of all $(\ga)_A$-linear morphisms of group functors~${(L^n\ga)_A\to(L^m\ga)_A}$ and let $\Hom^{\rm rg}\big((L^n\ga)_A,(L^m\ga)_A\big)$ be the set of all \mbox{$(\ga)_A$-linear} morphisms of ring functors~$(L^n\ga)_A\to(L^m\ga)_A$.

\begin{prop}\label{prop:funccont}
Evaluation at $A$ defines the bijection
$$
{\Hom\big((L^n\ga)_A,(L^m\ga)_A)}\stackrel{\sim}\longrightarrow \Hom^{\rm c}_A\big(\LL^n(A),\LL^m(A)\big)\,,\qquad \alpha\longmapsto\alpha(A)\,.
$$
\end{prop}
\begin{proof}
Lemma~\ref{lemma:basechar} implies directly that for any $(\ga)_A$-linear morphism of group functors $\alpha\colon(L^n\ga)_A\to (L^m\ga)_A$, the evaluation $\alpha(A)\colon \LL^n(A)\to\LL^m(A)$ at $A$ is continuous. Clearly, the map $\alpha(A)$ is $A$-linear. Thus the map in the proposition is well-defined.

By the same reason, for any $A$-algebra $B$, the evaluation $\alpha(B)$ belongs to $\Hom_B^{\rm c}\big(\LL^n(B),\LL^m(B)\big)$. By Lemma~\ref{cor:conv}$(ii)$, the homomorphism~$\alpha(B)$ is defined uniquely by the values $\alpha(B)(bt^l)=b\,\alpha(B)(t^l)$, where $b\in B$, $l\in\z^n$. Therefore $\alpha(B)$ is defined uniquely by $\alpha(A)$. This proves that the map in the proposition is injective.

Finally, we prove the surjectivity. Take a homomorphism $\phi\in\Hom_A^{\rm c}\big(\LL^n(A),\LL^m(A)\big)$. For any $\lambda\in\Lambda_n$, the countable set $\{t^l\}$, ${l\in\z^n_{\lambda}}$, tends to zero in~$\LL^n(A)$. Since $\phi$ is continuous, the countable set $\{\phi(t^l)\}$, $l\in\Lambda_n$, tends to zero in $\LL^m(A)$. Let~$B$ be an $A$-algebra. For an element $f=\sum\limits_{l\in\z^n}b_lt^l\in\LL^n(B)$, put ${\alpha(B)(f):=\sum\limits_{l\in\z^n}b_l\phi(t^l)}$, which is a convergent series in the topological group $\LL^m(B)$, because the countable set $\{\phi(t^l)\}$, $l\in\Lambda_n$, tends to zero in $\LL^m(B)$ as well. This defines a morphism ${\alpha\in\Hom\big((L^n\ga)_A,(L^m\ga)_A)}$. By Lemma~\ref{cor:conv}$(i)$, we have $\alpha(A)=\phi$. This finishes the proof.
\end{proof}

\medskip

Now let us consider homomorphisms that respect products of iterated Laurent series. We will use the following simple observation.

\begin{lemma}\label{lem:mult}
Suppose that a homomorphism $\phi\in\Hom^{\rm c}_A\big(\LL^n(A),\LL^m(A)\big)$ satisfies
$$
\phi(t^l)=\phi(t_1)^{l_1}\ldots\phi(t_n)^{l_n}
$$
for any element $l=(l_1,\ldots,l_n)\in\z^n$. Then $\phi$ is a homomorphism of rings.
\end{lemma}
\begin{proof}
The condition of the lemma implies that $\phi$ respects products of Laurent polynomials. Recall that the product with a fixed iterated Laurent series is continuous,~see~\cite[Lem.\,3.5(i)]{GOMS}, and that the set of Laurent polynomials is dense in $\LL^n(A)$ by Lemma~\ref{cor:conv}$(ii)$.
Applying this twice, we obtain first that $\phi$ respects products of Laurent polynomials with iterated Laurent series and then that $\phi$ respects products of arbitrary elements of~$\LL^n(A)$.
\end{proof}

Proposition~\ref{prop:funccont} implies the following fact.

\begin{corol}\label{cor:funccontrg}
Evaluation at $A$ defines the bijection
$$
{\Hom^{\rm rg}\big((L^n\ga)_A,(L^m\ga)_A\big)}\stackrel{\sim}\longrightarrow \Hom^{\rm c,alg}_A\big(\LL^n(A),\LL^m(A)\big)\,,\qquad \alpha\longmapsto\alpha(A)\,.
$$
\end{corol}
\begin{proof}
By Proposition~\ref{prop:funccont}, the map is well-defined and injective. Also by Proposition~\ref{prop:funccont}, for any ${\phi\in \Hom^{\rm c,alg}_A\big(\LL^n(A),\LL^m(A)\big)}$, there is $\alpha\in{\Hom\big((L^n\ga)_A,(L^m\ga)_A\big)}$ such that $\alpha(A)=\phi$. Since $\phi$ is a homomorphism of rings, it follows from Lemma~\ref{lem:mult} that for any $A$-algebra $B$, the continuous $B$-linear map ${\alpha(B)\colon \LL^n(B)\to\LL^m(B)}$ is a homomorphism of rings as well, whence we have $\alpha\in{\Hom^{\rm rg}\big((L^n\ga)_A,(L^m\ga)_A\big)}$.
\end{proof}

\begin{rmk}
If a ring $A$ is of zero characteristic, that is, the natural homomorphism $\z\to A$ is injective, then any endomorphism of the group functor~$(\ga)_A$ is $(\ga)_A$-linear (see, e.g.,~\cite[Lem.\,6.20(i)]{GOMS}). This implies that in this case, one can omit in Lemma~\ref{lemma:basechar}, Proposition~\ref{prop:funccont}, and Corollary~\ref{cor:funccontrg} the condition that morphisms of group functors (or ring functors) are $(\ga)_A$-linear (see the beginning of the proof of Lemma~\ref{lemma:basechar}).
\end{rmk}

\medskip

Here is an application of Proposition~\ref{prop:funccont} and Corollary~\ref{cor:funccontrg}. For an arbitrary ring~$R$, by~$\M_{m\times n}(R)$ denote the set of all $(m\times n)$-matrices with entries in $R$. Consider elements of $R^n$ as columns.

\begin{defin}\label{def:ups}
For any $\phi\in\Hom_A^{\rm c,alg}\big(\LL^n(A),\LL^m(A)\big)$, put
$$
\Upsilon(\phi):=\big(\nu(\phi(t_1)),\ldots,\nu(\phi(t_n))\big)\in\M_{m\times n}\big(\uz(A)\big)\,.
$$
\end{defin}

\begin{prop}\label{prop:valuation}
For any $\phi\in\Hom_A^{\rm c,alg}\big(\LL^n(A),\LL^m(A)\big)$ and any $f\in\LL^n(A)^*$, there is an equality in $\uz(A)^n$
$$
\nu\big(\phi(f)\big)=\Upsilon(\phi)\cdot \nu(f)\,.
$$
\end{prop}
\begin{proof}
By Corollary~\ref{cor:funccontrg}, there is a unique $\alpha\in\Hom^{\rm rg}\big((L^n\ga)_A,(L^m\ga)_A\big)$ such that~${\alpha(A)=\phi}$. Since $\alpha$ is a morphism of ring functors, $\alpha$ induces a morphism of group functors ${(L^n\gm)_A\to(L^m\gm)_A}$, which we denote also by $\alpha$ for simplicity.

Let us show that the morphism $\alpha\colon (L^n\gm)_A\to (L^m\gm)_A$ sends the group subfunctor $(L^n\gm)^0_A\subset (L^n\gm)_A$ to the group subfunctor $(L^m\gm)^0_A\subset (L^m\gm)_A$. Since the functor~$(L^n\gm)^0$ is represented by an absolutely connected ind-affine scheme over~$\z$, the group functor $(L^n\gm)^0_A$ is represented by an absolutely connected ind-affine scheme over $A$, see~\cite[Def.\,5.18]{GOMS} and~\cite[Prop.\,6.13(iv)]{GOMS}. It follows that any morphism of group functors $(L^n\gm)^0_A\to \uz_A$ over $A$ is equal to zero, see~\cite[Prop.\,5.23]{GOMS}. Therefore the composition
$$
(L^n\gm)^0_A\stackrel{\alpha}\longrightarrow (L^m\gm)_A\stackrel{\nu}\longrightarrow \uz^m_A
$$
is also equal to zero. Hence~$\alpha$ sends~$(L^n\gm)^0_A$ to $(L^m\gm)^0_A$.

We obtain that $\alpha$ defines a homomorphism of group functors from the quotient
$$
\nu\,:\,(L^n\gm)_A/(L^n\gm)_A^0\stackrel{\sim}\longrightarrow \uz^n_A
$$
to the quotient
$$
\nu\,:\,(L^m\gm)_A/(L^m\gm)_A^0\stackrel{\sim}\longrightarrow \uz^m_A\,,
$$
which is given by a matrix in $\M_{m\times n}\big(\uz(A)\big)$. It is easy to see that this matrix is nothing but $\Upsilon(\phi)$, which finishes the proof.
\end{proof}

\begin{corol}\label{cor:hommon}
When $m=n$, the map
$$
\Upsilon\,:\, \Hom^{\rm c,alg}_A\big(\LL^n(A),\LL^n(A)\big)\longrightarrow \M_{n\times n}\big(\uz(A)\big)
$$
is a homomorphism of monoids. In particular, if an endomorphism ${\phi\in\Hom^{\rm c,alg}_A\big(\LL^n(A),\LL^n(A)\big)}$ is invertible, then the matrix ${\Upsilon(\phi)\in\M_{n\times n}\big(\uz(A)\big)}$ is invertible as well.
\end{corol}

\section{Continuous homomorphisms and changes of parameters}

In this section, we describe all possible changes of parameters $t_1,\ldots,t_n$ under continuous homomorphisms of $A$-algebras $\phi\colon\LL^n(A)\to\LL^m(A)$, that is, we describe all possible collections $\big(\phi(t_1),\ldots,\phi(t_n)\big)$ of elements in $\LL^m(A)^*$, see Theorem~\ref{prop:contchange}. Such a collection determines uniquely the initial homomorphism~$\phi$.

\medskip

We will use the following elementary fact on matrices with integral entries.

\begin{lemma}\label{lem:upper}
For any matrix $M\in\M_{m\times n}(\z)$, the following conditions are equivalent:
\begin{itemize}
\item[(i)]
the map
$$
M\,:\,\z^n\longrightarrow \z^m\,,\qquad l\longmapsto M\cdot l\,,
$$
strictly preserves the lexicographical order, that is, if $l>0$, then $M\cdot l>0$;
\item[(ii)]
the matrix $M$ is in column echelon form with positive leading entries, that is, the matrix~$M$ has a form
$$
\begin{pmatrix}
x_{11}& x_{12}&\ldots  &x_{1n} \\
\vdots& \vdots& &\vdots \\
x_{p_1, 1}&\vdots&&\vdots\\
0&x_{p_2,2}&&\vdots&\\
\vdots& 0&&x_{p_n,n}\\
\vdots&\vdots&&0 \\
\vdots&\vdots&&\vdots \\
0&0&\ldots&0
\end{pmatrix}
$$
where $1\leqslant p_1<\ldots <p_n\leqslant m$ and $x_{p_i,i}>0$ for any $i$, $1\leqslant i\leqslant n$;
\item[(iii)]
the map $M\colon \z^n\to\z^m$ is injective and for any $\lambda\in\Lambda_n$, there is $\mu\in\Lambda_m$ such that $M\cdot \z^n_{\lambda}\subset \z^m_{\mu}$.
\end{itemize}
\end{lemma}
\begin{proof}
$(i)\Rightarrow (ii)$  The proof is by induction on $n$. The base $n=1$ is obvious. Let us make the induction step from $n-1$ to $n$, where $n\geqslant 2$. The inequality
${(0,\ldots,0,1,0)<(0,\ldots,0,0,1)}$ implies that either $p_n>p_{n-1}$ and $x_{p_n,n}>0$, or $p_n=p_{n-1}=p$ and ${x_{p,n}>x_{p,n-1}}$. Suppose the second case holds. Since $x_{p,n-1}\ne 0$ (actually, $x_{p,n-1}> 0$), there is $k\in\z$ such that $kx_{p,n-1}+x_{p,n}<0$. Then we obtain a contradiction, because ${(0,\ldots,0,k,1)>0}$ and $M\cdot (0,\ldots,0,k,1)^{\top}<0$.

$(ii)\Rightarrow (iii)$ It follows from the explicit form of the matrix $M$ in item~$(ii)$ that the rank of $M$ (over $\Q$) equals $n$. Thus the map ${M\colon\z^n\to\z^m}$ is injective.

Let us prove by induction on $m$ that the second condition in item~$(iii)$ also holds true. We consider the base $m=0$, in which case the statement is satisfied, being void. Let us make an induction step from $m-1$ to $m$, where $m\geqslant 1$.

Suppose that $p_n<m$. Then the image of the map $M\colon \z^n\to\z^m$ is contained in the subgroup ${\z^{m-1}=\{(*,\ldots,*,0)\}}$. By the induction hypothesis, there is $\kappa\in\Lambda_{m-1}$ such that $M\cdot \z^n_{\lambda}\subset\z^{m-1}_{\kappa}$. We let $\mu=(\mu',0)\in\Lambda_m$, where $\mu'\colon\z\to\Lambda_{m-1}$ is any map such that $\mu'(0)=\kappa$.

Now suppose that $p_n=m$. By definition, we have $\lambda=(\lambda',\lambda_n)$, where ${\lambda'\colon\z\to\Lambda_{n-1}}$ and $\lambda_n\in\z$. Put $\mu_m:=\lambda_nx_{mn}\in\z$. Let $M'$ be an \mbox{$(m-1)\times(n-1)$-matrix} which is equal to $M$ without the last row and the last column. By the induction hypothesis, for any $l\geqslant\lambda_n$, there is $\rho(l)\in\Lambda_{m-1}$ such that $M'\cdot \z^{n-1}_{\lambda'(l)}\subset\z^{m-1}_{\rho(l)}$. Now let $\mu'\colon\z\to\Lambda_{m-1}$ be a map such that for any $l\geqslant \lambda_n$, the element $\mu'(lx_{mn})\in\Lambda_{m-1}$ satisfies the condition
$$
\z^{m-1}_{\rho(l)}+l\cdot (x_{1n},\ldots,x_{m-1,n})\subset \z^{m-1}_{\mu'(lx_{mn})}\,.
$$
Then we obtain $M\cdot\z^n_{\lambda}\subset\z^m_{\mu}$, where $\mu=(\mu',\mu_m)$.

$(iii)\Rightarrow (i)$ Note that an element $l\in\z^n$ satisfies $l>0$ if and only if $l\ne 0$ and there is $\lambda\in\Lambda_n$ such that the set ${\{kl\mid k\in\N\}}$ is contained in $\z^n_{\lambda}$ (this can be shown easily by induction on $n$). This gives the needed implication.
\end{proof}

\begin{defin}\label{def:plus}
Denote by
$$
\M_{m\times n}^+(\z)\subset \M_{m\times n}(\z)
$$
the set of all matrices that satisfy the equivalent conditions of Lemma~\ref{lem:upper}.
Similarly, let~$\M^+_{m\times n}\big(\uz(A)\big)$ be the set of all \mbox{$(m\times n)$-matrices} with entries in $\uz(A)$ such that point-wise on $\Spec(A)$ these maitrices belong to~$\M^+_{m\times n}(\z)$.
\end{defin}

Note that $n\leqslant m$ whenever $\M_{m\times n}^+(\z)$ is non-empty.

\medskip

Let us say that a set $\{f_i\}$, $i\in \N$, of iterated Laurent series in $\LL^n(A)$ is {\it bounded} if there is $\lambda\in\Lambda_n$ such that for all $i\in\N$, we have ${\supp(f_i)\subset \z^n_{\lambda}}$. Clearly, the pair-wise product of two bounded sets is also bounded.

\begin{lemma}\label{lem:zerobound}
\hspace{0cm}
\begin{itemize}
\item[(i)]
If a countable set $\{f_i\}$, $i\in\N$, of iterated Laurent series tends to zero, then this set is bounded.
\item[(ii)]
If a countable set $\{f_i\}$, $i\in\N$, of iterated Laurent series tends to zero and a countable set $\{g_i\}$, $i\in\N$, of iterated Laurent series is bounded, then the diagonal product~$\{f_ig_i\}$, $i\in\N$, tends to zero.
\end{itemize}
\end{lemma}
\begin{proof}
$(i)$
Take an element $\mu\in\Lambda_n$. Note that for any $i\in \N$, the intersection ${\supp(f_i)\cap (-\z^n_{\mu})}$ is finite and this intersection is empty for all but finitely many $i\in\N$. This implies finiteness of the intersection $X\cap(-\z^n_{\mu})$, where $X:=\bigcup\limits_{i\in \N}\supp(f_i)$. Now, the statement follows from Lemma~\ref{lemma:intersect}.

$(ii)$
Let $\lambda\in\Lambda_n$ be an element such that $\supp(g_i)\subset \z^n_{\lambda}$ for all $i\in\N$. Given an element $\mu\in\Lambda_n$, take an element $\rho\in\Lambda_n$ such that ${\z^n_{\lambda}+\z^n_{\mu}\subset\z^n_{\rho}}$. Since the countable set $\{f_i\}$, $i\in\N$, tends to zero, all but finitely many $i\in\N$ satisfy $0\notin\supp(f_i)+\z^n_{\rho}$. The embeddings
$$
\supp(f_i)+\z^n_{\rho}\supset \supp(f_i)+\z^n_{\lambda}+\z^n_{\mu}\supset \supp(f_ig_i)+\z^n_{\mu}
$$
imply that all but finitely many $i\in\N$ satisfy $0\notin \supp(f_ig_i)+\z^n_{\mu}$.
\end{proof}

\begin{lemma}\label{lemma:coundpower}
For any element $f\in\LL^n(A)^*$ with $\nu(f)=0$, the set $\{f^k\}$, $k\in\Z$, is bounded.
\end{lemma}
\begin{proof}
We have a decomposition $f=c\cdot (1+g)$, where $c\in A^*$ and $g=\sum\limits_{l\in\z^n}a_lt^l$ is an element such that $a_0=0$ and the iterated Laurent series $\sum\limits_{l<0}a_lt^l$ is nilpotent. Clearly, for any integer $k\in\z$, we have that $\supp(f^k)$ is contained in the union $\bigcup\limits_{i\in\N}\supp(g^i)$. By~\cite[Prop.\,3.8]{GOMS}, the countable set $\{g^i\}$, $i\in\N$, tends to zero. Thus by Lemma~\ref{lem:zerobound}$(i)$, the set~$\{g^i\}$, $i\in\N$, is bounded, which proves the lemma.
\end{proof}

\begin{defin}
Denote by
$$
\Hb_{m,n}(A)\subset \big(\LL^m(A)^*\big)^{\times n}
$$
the set of all collections $(\varphi_1,\ldots,\varphi_n)$ of invertible iterated Laurent series in $t_1,\ldots,t_m$ such that the matrix ${\big(\nu(\varphi_1),\ldots,\nu(\varphi_n)\big)}\in\M_{m\times n}\big(\uz(A)\big)$ belongs to~$\M^+_{m\times n}\big(\uz(A)\big)$.
\end{defin}

\begin{prop}\label{prop:converg}
A collection ${\varphi=(\varphi_1,\ldots,\varphi_n)\in \big(\LL^m(A)^*\big)^{\times n}}$
belongs to $\Hb_{m, n}(A)$ if and only if for any $\lambda\in\Lambda_n$, the countable set $\{\varphi^l\}$, $l\in\z^n_{\lambda}$, tends to zero in~$\LL^m(A)$.
\end{prop}
\begin{proof}
Taking a decomposition of $A$ into a finite product of rings, we may assume that all $\nu(\varphi_i)\in\uz^n(A)$, $1\leqslant i\leqslant n$, are constant, that is, are elements of~$\z^n$.

Suppose that $\varphi\in\Hb_{m,n}(A)$. Put $M:={\big(\nu(\varphi_1),\ldots,\nu(\varphi_n)\big)}\in\M^+_{m\times n}(\z)$. For each $i$, $1\leqslant i\leqslant n$, let $f_i\in\LL^m(A)^*$ be the element that satisfies
$$
\varphi_i=t^{\nu(\varphi_i)}\cdot f_i\,.
$$
In particular, we have $\nu(f_i)=0$. Then for any $l\in\z^n$, there is an equality $\varphi^l=t^{M\cdot l}\cdot f^l$, where $f = (f_1, \ldots, f_n)$.

Since the matrix $M$ satisfies condition~$(iii)$ of Lemma~\ref{lem:upper}, we have an embedding $M\colon \z^n_{\lambda}\hookrightarrow\z^m_{\mu}$ for some $\mu\in\Lambda_m$. For any $\rho\in\Lambda_m$, the intersection $\z^m_{\mu}\cap (-\z^m_{\rho})$ is finite. This implies that the countable set $\{t^{M\cdot l}\}$, $l\in \z^n_{\lambda}$, tends to zero.

It follows from Lemma~\ref{lemma:coundpower} that the countable set $\{f^l\}$, where $l\in \z^n$, is bounded. Thus by Lemma~\ref{lem:zerobound}$(ii)$, the countable set $\{t^{M\cdot l}\cdot f^l\}$, $l\in\z^n_{\lambda}$, tends to zero.

Now suppose that the countable set $\{\varphi^l\}$, $l\in\z^n_{\lambda}$, tends to zero. Then by Lemma~\ref{lem:zerobound}$(i)$, this set is bounded. On the other hand, we have that $\nu(\varphi^l)\in\supp(\varphi^l)$ and ${\nu(\varphi^l)=M\cdot l\in\z^m}$. Thus there is $\mu\in\Lambda_m$ such that $M\cdot\z^n_{\lambda}\subset \z^m_{\mu}$.

Suppose that the map $M\colon\z^n\to\z^m$ has a non-zero kernel. Then there is $\lambda\in\Lambda_n$ such that for infinitely many $l\in \z^n_{\lambda}$, we have $M\cdot l=0$. Consequently the countable set~$\{\varphi^l\}$, $l\in\z^n_{\lambda}$, contains infinitely many elements $g$ with $\nu(g)=0$. This contradicts the condition that $\{\varphi^l\}$, $l\in\z^n_{\lambda}$, tends to zero. Thus the map $M\colon\z^n\to\z^m$ is injective and~$M$ satisfies condition~$(iii)$ of Lemma~\ref{lem:upper}.
\end{proof}

Now we are ready to prove the main result of this section.

\begin{theor}\label{prop:contchange}
We have a bijection (see Definition~\ref{def:hom})
$$
\Hom^{\rm c,alg}_{A}\big(\LL^n(A),\LL^m(A)\big)\stackrel{\sim}\longrightarrow \Hb_{m,n}(A)\,,\qquad \phi\longmapsto\big(\phi(t_1),\ldots\phi(t_n)\big)\,.
$$
\end{theor}
\begin{proof}
Consider an element $\phi\in\Hom^{\rm c,alg}_{A}\big(\LL^n(A),\LL^m(A)\big)$. For any $\lambda\in\Lambda_n$, the countable set $\{t^l\}$, $l\in\z^n_{\lambda}$, tends to zero in $\LL^n(A)$. Therefore the countable set $\{\phi(t^l)\}$, $l\in\z^n_{\lambda}$, tends to zero in $\LL^m(A)$. For any $l\in\z^n$, we have an equality $\phi(t^l)=\varphi^l$, where $\varphi_i:=\phi(t_i)$, $1\leqslant i\leqslant n$, and $\varphi=(\varphi_1,\ldots,\varphi_n)$. Thus by Proposition~\ref{prop:converg}, the map in the theorem is well-defined. The injectivity of this map follows directly from Lemma~\ref{cor:conv}$(ii)$.

Let us prove the surjectivity. Take a collection $\varphi=(\varphi_1,\ldots,\varphi_n)\in \Hb_{m,n}(A)$. By Proposition~\ref{prop:converg}, for any $\lambda\in\Lambda_n$, the countable set~$\{\varphi^l\}$, $l\in\z^n_{\lambda}$, tends to zero. Hence the series $\sum\limits_{l\in\z^n_{\lambda}}a_l\varphi^l$ converges in~$\LL^m(A)$ and we have a well-defined map
$$
\phi\,:\,\LL^n(A)\longrightarrow \LL^m(A)\,,\qquad \mbox{$f=\sum\limits_{l\in\z^n_{\lambda}}a_lt^l\longmapsto f(\varphi_1,\ldots,\varphi_n)=\sum\limits_{l\in\z^n_{\lambda}}a_l\varphi^l$}\,.
$$
Clearly, the map $\phi$ is $A$-linear, one has $\phi(t_i)=\varphi_i$, $1\leqslant i\leqslant n$, and one easily shows that $\phi$ is a homomorphism of additive groups. Similarly, for any $A$-algebra $B$, the collection $(\varphi_1,\ldots,\varphi_n)$ defines a $B$-linear homomorphism of additive groups ${\LL^n(B)\to\LL^m(B)}$. This gives a $(\ga)_A$-linear morphism of group functors~$\alpha\colon (L^n\ga)_A\to(L^m\ga)_A$ such that $\alpha(A)=\phi$. Therefore by Proposition~\ref{prop:funccont}, the map~$\phi$ is continuous. Finally, Lemma~\ref{lem:mult} implies that $\phi$ is a homomorphism of rings. Thus $\phi$ is a well-defined element in $\Hom^{\rm c,alg}_A\big(\LL^n(A),\LL^m(A)\big)$.
\end{proof}

\begin{corol} \label{cor:inequality}
If $n>m$, then there are no continuous homomorphisms of $A$-algebras from $\LL^n(A)$ to~$\LL^m(A)$.
\end{corol}

\begin{corol}\label{cor:section}
The homomorphism of monoids (see Corollary~\ref{cor:hommon})
$$
\Upsilon\,:\,\Hom^{\rm c,alg}_A\big(\LL^n(A),\LL^n(A)\big)\longrightarrow \M_{n\times n}^+\big(\uz(A)\big)
$$
has a natural monoid section that sends a matrix ${M\in\M_{n\times n}^+\big(\uz(A)\big)}$ to the endomorphism~$\phi$ such that $\phi(t_i)=t^{l_i}$, $1\leqslant i\leqslant n$, where $M=(l_1,\ldots,l_n)$ and $l_i\in\z^n$.
\end{corol}

\medskip

Theorem~\ref{prop:contchange} also implies that one can define an associative composition $\circ$ of collections from~$\Hb_{p,m}$ and $\Hb_{m,n}$ as follows.

\begin{corol}\label{cor:monoidphi}
For all collections $(\varphi_1,\ldots,\varphi_n)\in\Hb_{m,n}(A)$ and $(\vartheta_1,\ldots,\vartheta_m)\in\Hb_{p,m}(A)$, the composition
$$
(\vartheta_1,\ldots,\vartheta_m)\circ(\varphi_1,\ldots,\varphi_n):=\big(\varphi_1(\vartheta_1,\ldots,\vartheta_m),\ldots,\varphi_n(\vartheta_1,\ldots,\vartheta_m)\big)\in \big(\LL^p(A)\big)^{\times n}
$$
belongs to $\Hb_{p,n}(A)\subset \big(\LL^p(A)^*\big)^{\times n}\subset \big(\LL^p(A)\big)^{\times n}$. Moreover, there is an equality of matrices in $\M_{p\times n}\big(\uz(A)\big)$
$$
\big(\nu(\varphi_1(\vartheta_1,\ldots,\vartheta_m)),\ldots,\nu(\varphi_n(\vartheta_1,\ldots,\vartheta_m))\big)=
\big(\nu(\vartheta_1),\ldots,\nu(\vartheta_m)\big)\cdot \big(\nu(\varphi_1),\ldots,\nu(\varphi_n)\big)\,.
$$
\end{corol}
\begin{proof}
By Theorem~\ref{prop:contchange}, the collections $(\varphi_1,\ldots,\varphi_n)$ and $(\vartheta_1,\ldots,\vartheta_m)$ correspond to homomorphisms~$\phi\in\Hom_A^{\rm c,rg}\big(\LL^n(A),\LL^m(A)\big)$ and~$\theta\in\Hom_A^{\rm c,rg}\big(\LL^m(A),\LL^p(A)\big)$, respectively. It follows that for any $i$ such that $1\leqslant i\leqslant n$, there are equalities
$$
(\theta\circ\phi)(t_i)=\theta(\varphi_i)=\varphi_i(\vartheta_1,\ldots,\vartheta_m)\,.
$$
Hence the composition $\theta\circ\phi\in\Hom_A^{\rm c,rg}(\LL^n(A),\LL^p(A))$ corresponds to the collection $(\vartheta_1,\ldots,\vartheta_m)\circ(\varphi_1,\ldots,\varphi_n)$. Therefore again by Theorem~\ref{prop:contchange}, we see that the collection ${(\vartheta_1,\ldots,\vartheta_m)\circ(\varphi_1,\ldots,\varphi_n)}$ belongs to $\Hb_{p,n}(A)$. The equality between matrices follows from Proposition~\ref{prop:valuation}.
\end{proof}

\section{Representability of functors and invariance of the residue}

In this section, we show representability for functors defined by the sets of continuous homomorphisms, see Proposition~\ref{prop:repr} and Corollary~\ref{cor:repr}. Then we investigate how the residue map is changed under continuous homomorphisms, see Proposition~\ref{prop:invres} and Corollary~\ref{cor:res}.

\medskip

Define the functor on the category of rings (see Definition~\ref{def:hom})
$$
\Hc om^{\rm c,alg}(\LL^n,\LL^m)\,:\,A\longmapsto \Hom^{\rm c,alg}_A\big(\LL^n(A),\LL^m(A)\big)\,.
$$
This is indeed a functor by Corollary~\ref{cor:funccontrg} (cf. Remark~\ref{rmk:funvtadj} below). The functor (see Definition~\ref{def:plus})
$$
\M^+_{m\times n}(\uz)\,:\,A\longmapsto \M^+_{m\times n}\big(\uz(A)\big)
$$
is clearly represented by an ind-scheme, which is a formal direct limit of finite disjoint unions of $\Spec(\z)$.
Theorem~\ref{prop:contchange} implies directly the following representability result.

\begin{prop}\label{prop:repr}
The functor $\Hc om^{\rm c,alg}(\LL^n,\LL^m)$ is represented by the ind-affine scheme over $\z$
$$
\M_{m\times n}^+\big(\uz\big)\times\big((L^m\gm)^0\big)^{\times n}\,,
$$
where for any ring $A$, the corresponding bijection sends $\phi\in\Hom^{\rm c,alg}_A\big(\LL^n(A),\LL^m(A)\big)$ to
$$
\big(\Upsilon(\phi),(f_1,\ldots,f_n)\big)\in \M_{m\times n}^+\big(\uz(A)\big)\times\big((L^m\gm)^0\big)^{\times n}(A)\,,
$$
where $f_i:=\phi(t_i)\cdot t^{-\nu(\phi(t_i))}$, $1\leqslant i\leqslant n$ (see Definition~\ref{def:ups} for $\Upsilon(\phi)$).
\end{prop}

Recall that for $m=n$, by Corollary~\ref{cor:hommon}, we have a morphism of monoid functors
$$
\Upsilon\,:\,\Hc om^{\rm c,alg}(\LL^n,\LL^n)\longrightarrow \M_{n\times n}^+\big(\uz\big)\,.
$$
Denote the kernel of this morphism by $\Ker(\Upsilon)_n$. Define the subfunctor
$$
\Ker(\Upsilon)^{\rm nil}_n\subset\Ker(\Upsilon)_n
$$
that sends a ring $A$ to the set of all $\phi\in\Hom^{\rm c,alg}_A\big(\LL^n(A),\LL^n(A)\big)$ such that for any~$i$ with ${1\leqslant i\leqslant n}$, we have
${\phi(t_i)=t_i+h_i}$, where $h_i\in \LL^n(A)$ is a Laurent polynomial (not merely a series) with nilpotent coefficients.

\medskip

By definition, a regular function on a functor $F$ on the category of $A$-algebras is a morphism of functors from $F$ to the affine line $(\Ab^1)_A$ over $A$. By $\OO(F)$ denote the $A$-algebra of all regular functions on $F$.

Proposition~\ref{prop:repr} implies many useful properties of regular functions on the functor~$\Hc om^{\rm c,alg}(\LL^n,\LL^m)$. Corollary~\ref{cor:repr} below contains two of them. The first one is that in order to check equalities between regular functions on the functor ${\Hc om^{\rm c,alg}(\LL^n,\LL^m)\times L^n\ga}$, it is enough to consider their evaluations at $\Q$-algebras only. The second property is that in order to check equalities between regular functions on the functor $\Ker(\Upsilon)_n\times L^n\ga$, it is enough to consider for an arbitrary ring $A$ only elements from ${\Ker(\Upsilon)_n^{\rm nil}(A)\times\LL^n(A)\subset \Ker(\Upsilon)_n(A)\times\LL^n(A)}$, which are much simpler to be treated.

These properties of regular functions are of utmost importance for what follows. We warn the reader that the proof of Corollary~\ref{cor:repr} relies heavily on the theory of thick ind-cones developed in~\cite{GOMS}. However, in this paper, the theory of thick ind-cones is used explicitly only here.

\begin{corol}\label{cor:repr}
\hspace{0cm}
\begin{itemize}
\item[(i)]
The natural ring homomorphism
$$
\OO\big(\Hc om^{\rm c,alg}(\LL^n,\LL^m)\times L^n\ga\big)\longrightarrow\OO\big(\Hc om^{\rm c,alg}(\LL^n,\LL^m)_{\Q}\times (L^n\ga)_{\Q}\big)
$$
is injective.
\item[(ii)]
The ring homomorphism which is the restriction map
$$
\OO\big(\Ker(\Upsilon)_n\times L^n\ga\big)\longrightarrow \OO\big(\Ker(\Upsilon)_n^{\rm nil}\times L^n\ga\big)
$$
is injective.
\end{itemize}
\end{corol}
\begin{proof}
$(i)$
By Proposition~\ref{prop:repr}, we have an isomorphism
$$
\Hc om^{\rm c,alg}(\LL^n,\LL^m)\times L^n\ga\stackrel{\sim}\longrightarrow\M_{m\times n}^+\big(\uz\big)\times\big((L^m\gm)^0\big)^{\times n}\times L^n\ga\,.
$$
Clearly, the ind-affine scheme $\M_{m\times n}^+\big(\uz\big)$ is ind-flat over $\z$, that is, this scheme is a formal direct limit of flat affine schemes over $\z$. Further, the ind-affine scheme~$\big((L^n\gm)^0\big)^{\times m}$ is isomorphic to a product of a thick ind-cone and an  ind-flat scheme over~$\z$, see~\cite[Def.\,5.7, Def.\,5.10]{GOMS} and~\cite[Prop.\,6.8(iii), Lem.\,5.13]{GOMS}. The ind-affine scheme $L^n\ga$ is an ind-affine space by~\cite[Ex.\,6.7]{GOMS}, whence $L^n\ga$ is ind-flat over $\z$. Finally, by~\cite[Prop.\,5.17]{GOMS}, for any $X$ which is a product of a thick ind-cone and an ind-flat scheme over~$\z$, the homomorphism $\OO(X)\to\OO(X_{\Q})$ is injective.

$(ii)$ By Proposition~\ref{prop:repr}, we have an isomorphism
$$
\Ker(\Upsilon)_n\times L^n\ga\stackrel{\sim}\longrightarrow \big((L^n\gm)^0\big)^{\times n}\times L^n\ga\,.
$$
By~\cite[Prop.\,6.13(iii)]{GOMS}, the ind-scheme $\big((L^n\gm)^0\big)^{\times n}$ contains an absolutely dense ind-closed subscheme $\big((L^n\gm)^{\sharp}\big)^{\times n}$, see~\cite[Def.\,5.24]{GOMS}. Here, for any ring $A$, the set (actually, the group)~${(L^n\gm)^{\sharp}(A)}$ consists of all elements $1+\sum\limits_{l\in\z^n}a_lt^l\in(L^n\gm)(A)$ such that $\sum\limits_{l\leqslant 0}a_l t^l$ is a nilpotent element of $\LL^n(A)$, see~\cite[Lem.\,4.7(iii)]{GOMS}. Hence by~\cite[Lem.\,5.25]{GOMS}, the restriction map
$$
\OO\big(\big((L^n\gm)^0\big)^{\times n}\times L^n\ga\big)\longrightarrow \OO\big(\big((L^n\gm)^{\sharp}\big)^{\times n}\times L^n\ga\big)
$$
is injective.

Clearly, $\Ker(\Upsilon)^{\rm nil}_{n}$ is a subfunctor of $\big((L^n\gm)^{\sharp}\big)^{\times n}$,
 see Proposition~\ref{prop:repr}.
 Further, the ind-scheme $(L^n\gm)^{\sharp}$ is a thick ind-cone by~\cite[Prop.\,6.8(ii)]{GOMS}. Hence by~\cite[Lem.\,5.13]{GOMS} and~\cite[Ex.\,6.7]{GOMS}, the ind-scheme ${\big((L^n\gm)^{\sharp}\big)^{\times n}\times L^n\ga}$ is a thick ind-cone as well. It follows from~\cite[Prop.\,5.15]{GOMS} that the restriction map
$$
\OO\big(\big((L^n\gm)^{\sharp}\big)^{\times n}\times L^n\ga\big)\longrightarrow \OO\big(\Ker(\Upsilon)^{\rm nil}_n\times L^n\ga\big)
$$
is injective. This finishes the proof.
\end{proof}

\medskip

Now we describe how the residue map is changed under continuous homomorphisms. Untill the end of this section, we fix a continuous homomorphism of $A$-algebras ${\phi\colon\LL^n(A)\to\LL^m(A)}$. By Theorem~\ref{prop:contchange}, the matrix $\Upsilon(\phi)$ belongs to $\M^+_{m\times n}\big(\uz(A)\big)$. In particular, we have that $n\leqslant m$.

Given two locally constant functions $p,q\in\uz(A)$ on $\Spec(A)$, we say that $p<q$ or $p\leqslant q$ if this holds point-wise on $\Spec(A)$. Let
$$
1\leqslant p_1<\ldots<p_n\leqslant m \quad  \mbox{and}  \quad x_{p_i,i}>0 \quad \mbox{for}\quad 1\leqslant i\leqslant n\,,
$$
be the elements of $\uz(A)$ that correspond to the matrix $\Upsilon(\phi)\in\M^+_{m\times n}\big(\uz(A)\big)$ as in condition~$(ii)$ of Lemma~\ref{lem:upper}. Let
$$
1\leqslant q_1<\ldots<q_{m-n}\leqslant m
$$
be a collection of elements of $\uz(A)$ which is complementary to $(p_1,\ldots,p_n)$, that is, for all $1\leqslant i,j\leqslant n$, we have~$p_i\ne q_j$ point-wise on $\Spec(A)$ (in other words, for any point of $\Spec (A)$, the values of $p_i$ and $q_j$ at this point are not equal). Let $\sgn(\phi)\in\uz(A)$ be the locally constant function of sign of the locally constant permutation that sends $(1,2,\ldots,n)$ to~$(p_1,\ldots,p_m,q_{1},\ldots,q_{m-n})$. Also, put
$$
0<d(\phi):=\prod\limits_{i=1}^n x_{p_i,i}\in\uz(A)\,.
$$
For example, if $m=n$, then $\sgn(\phi)=1$ and $d(\phi)=\det\big(\Upsilon(\phi)\big)$.

The continuous homomorphism $\phi$ induces additive homomorphisms $\widetilde{\Omega}^i_{\LL^n(A)}\to \widetilde{\Omega}^i_{\LL^m(A)}$, $i\geqslant 0$, which we denote also by $\phi$ for simplicity.

\begin{prop}\label{prop:invres}
For any differential form $\omega\in\widetilde{\Omega}^n_{\LL^n(A)}$, there is an equality
\begin{equation}\label{eq:res}
\res\left(\phi(\omega)\wedge\frac{dt_{q_1}}{t_{q_1}}\wedge\ldots\wedge\frac{dt_{q_{m-n}}}{t_{q_{m-n}}}\right)=\sgn(\phi)d(\phi)\res(\omega)\,.
\end{equation}
\end{prop}
\begin{proof}
Both sides of formula~\eqref{eq:res} are functions with values in $A$ that depend on a homomorphism $\phi\in\Hc om^{\rm c,alg}(\LL^n,\LL^m)(A)$ and a differential form $\omega\in\widetilde{\Omega}^n_{\LL^n(A)}$. Clearly, the functor $A\mapsto \widetilde{\Omega}^n_{\LL^n(A)}$ is isomorphic to $L^n\ga$. Thus both sides of formula~\eqref{eq:res} are regular functions from ${\OO\big(\Hc om^{\rm c,alg}(\LL^n,\LL^m)\times L^n\ga\big)}$. Hence by Corollary~\ref{cor:repr}$(i)$, it is enough to prove the equality between the images of  these functions in the ring  ${\OO\big(\Hc om^{\rm c,alg}(\LL^n,\LL^m)_{\Q}\times (L^n\ga)_{\Q}\big)}$. In other words, it is enough to prove the proposition when $A$ is a $\Q$-algebra, which we assume from now on.

The left hand side of formula~\eqref{eq:res} is equal to zero if $\omega$ is exact. Therefore the left hand side of formula~\eqref{eq:res} with fixed $\phi$ defines an $A$-linear map from the quotient
$$
\widetilde{\Omega}^n_{\LL^n(A)}/d\widetilde{\Omega}^{n-1}_{\LL^n(A)}\stackrel{\sim}\longrightarrow A
$$
to $A$. Thus this map sends any differential form $\omega\in\widetilde{\Omega}^n_{\LL^n(A)}$ to $c\cdot\res(\omega)$, where $c\in A$ can be calculated for a particular differential form $\frac{dt_1}{t_1}\wedge\ldots\wedge\frac{dt_n}{t_n}$. Explicitly, we have
$$
c=\res\left(\phi\left(\frac{dt_1}{t_1}\wedge\ldots\wedge\frac{dt_n}{t_n}\right)\wedge\frac{dt_{q_1}}{t_{q_1}}\wedge\ldots\wedge\frac{dt_{q_{m-n}}}{t_{q_{m-n}}}\right)=
$$
$$
=\res\left(\frac{d\phi(t_1)}{\phi(t_1)}\wedge\ldots\wedge\frac{d\phi(t_n)}{\phi(t_n)}\wedge\frac{dt_{q_1}}{t_{q_1}}\wedge\ldots\wedge\frac{dt_{q_{m-n}}}{t_{q_{m-n}}}\right)\,.
$$
By~\cite[Prop.\,8.4]{GOMS}, we have that $c$ is the image in $A$ of the determinant of the \mbox{$(m\times m)$-matrix}
$$
\big(\Upsilon(\phi),\nu(t_{q_1}),\ldots,\nu(t_{q_{m-n}})\big)\in\M_{m\times m}\big(\uz(A)\big)\,.
$$
Obviously, this determinant equals $\sgn(\phi)d(\phi)$.
\end{proof}

\begin{corol}\label{cor:res}
Suppose that $m=n$. Then for any differential form $\omega\in\widetilde{\Omega}^n_{\LL^n(A)}$, we have
$$
\res\big(\phi(\omega)\big)=d(\phi)\res(\omega)\,.
$$
\end{corol}

\begin{rmk}\label{rmk:invres}
If $\phi$ is invertible, then by Corollary~\ref{cor:hommon} and Theorem~\ref{prop:contchange}, we have the equality $d(\phi)=1$. Thus by Corollary~\ref{cor:res}, for any differential form $\omega\in\widetilde{\Omega}^n_{\LL^n(A)}$, we have
$$
\res\big(\phi(\omega)\big)=\res(\omega)\,.
$$
\end{rmk}

\medskip

Proposition~\ref{prop:invres} has also the following application for injectivity of homomorphisms.

\begin{corol}\label{cor:inj}
Suppose that the image of $d(\phi)$ under the natural homomorphism $\uz(A)\to A$ is not a zero divisor in the ring $A$. Then the continuous homomorphism ${\phi\colon\LL^n(A)\to\LL^m(A)}$ is injective.
\end{corol}
\begin{proof}
Let us show that for any non-zero element $f\in\LL^n(A)$, its image $\phi(f)$ is also non-zero. Consider an element $l\in\z^n$ such that the $l$-th coefficient of $f$ is non-zero. Define a differential form
$$
\omega:=f t^{-l}\,\frac{dt_1}{t_1}\wedge\ldots\wedge\frac{dt_n}{t_n}\in\widetilde{\Omega}^n_{\LL^n(A)}\,.
$$
By construction, we have that $\res(\omega)\ne 0$. Therefore by the condition of the corollary and by Proposition~\ref{prop:invres}, we have (in the notation of the proposition) that
${\res\left(\phi(\omega)\wedge\frac{dt_{q_1}}{t_{q_1}}\wedge\ldots\wedge\frac{dt_{q_{m-n}}}{t_{q_{m-n}}}\right)\ne 0}$. Hence $\phi(\omega)\ne 0$, and
therefore  $\phi(f)\ne 0$.
\end{proof}

\begin{rmk}\label{rmk:inj}
Suppose that $\phi(f)=0$ for an element $f\in\LL^n(A)$. Then all coefficients of~$f$ are nilpotent elements of $A$, which can be proven as follows. First one shows that for any closed (with respect to the topology on $\LL^n(A)$) prime ideal $I\subset\LL^n(A)$, there is a prime ideal $\p\subset A$ such that $I$ consists of all iterated Laurent series with coefficients in~$\p$. For this, one uses that if $A$ is a domain, then the localization ${(A\smallsetminus\{0\})^{-1}\cdot\LL^n(A)}$ is a field. Now one sees that any continuous homomorphism of \mbox{$A$-algebras} $\phi\colon\LL^n(A)\to\LL^m(A)$ induces a bijection between the sets of closed prime ideals, because these ideals are defined by prime ideals of $A$ as above. Hence $f$ belongs to the intersection of all closed prime ideals in~$\LL^n(A)$, as the analogous holds for the element $\phi(f)=0$ in $\LL^m(A)$. Finally, one uses that the intersection of all prime ideals in $A$ coincides with the nilradical of $A$.
\end{rmk}

\smallskip

It is an interesting question whether there exists a non-injective continuous homomorphism of $A$-algebras $\phi\colon\LL^n(A)\to\LL^m(A)$. By Corollary~\ref{cor:inj} and Remark~\ref{rmk:inj}, we see that for such $\phi$, the image of  $d(\phi) \in \uz(A)$ in $A$ should be a zero divisor in $A$, and if $\phi(f)=0$ for $f \in \LL^n(A)$, then all coefficients of $f$ are nilpotent elements of~$A$.

\section{Invertibility criterion}

We start this section with a self-duality of the topological $A$-module $\LL^n(A)$, see Proposition~\ref{prop:seldual}. Then we study the adjoint map $\phi^{\vee}$ to the map of $\LL^n(A)$-modules of top differential forms, where the last map is induced by a continuous homomorphism $\phi$, see Proposition~\ref{lem:leftinverse} and Remark~\ref{rmk:expladj}. Finally, we prove in Theorem~\ref{theor:inv} a criterion of invertibility of continuous endomorphisms of the $A$-algebra~$\LL^n(A)$.

\medskip

By $\LL^n(A)^{\vee}$ denote the \mbox{$A$-module} of all continuous $A$-linear maps $\LL^n(A)\to A$, where we consider the discrete topology on $A$. The $A$-module $\LL^n(A)^{\vee}$ has a natural topology such that $\LL^n(A)^{\vee}$ a topological group. The base of open neighborhoods of zero in $\LL^n(A)^{\vee}$ is given by annihilators of compact subsets of $\LL^n(A)$.

Given an iterated Laurent series~$f \in \LL^n(A)$  and an element $k\in\z^n$, by $[f]_k  \in A$ we denote the $k$-th coefficient of~$f$.

\begin{prop}\label{prop:seldual}
For any element $k\in\z^n$, the pairing
$$
\LL^n(A)\times\LL^n(A)\longrightarrow A\,,\qquad (f,g)\longmapsto [fg]_k\,,
$$
defines an isomorphism of topological $A$-modules
$$
\tau\,:\,\LL^n(A)\stackrel{\sim}\longrightarrow \LL^n(A)^{\vee}\,,\qquad f\longmapsto\big(g\mapsto [fg]_k\big)\,.
$$
\end{prop}
\begin{proof}
Recall that the product with a fixed iterated Laurent series is a continuous map. Applying to one of the arguments in the pairing the continuous automorphism given by the product with $t^k$, we see that it is enough to prove the proposition when $k=0$, which we assume from now on.

Since the map $\LL^n(A)\to A$, $f\mapsto [f]_0$, is continuous, $\tau$ is a well-defined $A$-linear map.

If $m\in\z^n$ is  an element such that the $m$-th coefficient of an iterated Laurent series $f\in\LL^n(A)$ is non-zero, then for $g:=t^{-m}$, we have that $[fg]_0\ne 0$. Hence the map~$\tau$ is injective.

Let $\chi\colon \LL^n(A)\to A$ be a continuous $A$-linear map. Then there is $\lambda\in\Lambda_m$ such that we have~$\chi(U_{\lambda})=0$. Now we put $a_l:=\chi(t^{-l})$ for $l \in \z^n$. It follows from the definition of $U_{\lambda}$ that $a_l=0$ if ${-l\notin (-\z^n_{\lambda})}$, that is, if ${l\notin \z^n_{\lambda}}$. Hence $f:=\sum\limits_{l\in\z^n}a_lt^l$ is a well-defined element of $\LL^n(A)$ with $\supp(f)\subset\z^n_{\lambda}$. We claim that for any $g\in\LL^n(A)$, we have $\chi(g)=[fg]_0$. Indeed, by continuity and $A$-linearity, it is enough to check this when $g$ is a monomial, which holds by construction of $f$. Thus we have shown that the map~$\tau$ is surjective.

Finally, let us prove that $\tau$ is a homeomorphism. Given an element $\lambda\in\Lambda_n$, by $K_{\lambda}$ denote the $A$-submodule of~$\LL^n(A)$ that consists of all elements $f\in\LL^n(A)$ with condition $\supp(f)\subset\z^n_{\lambda}$. Note that $K_{\lambda}$ is the closure in $\LL^n(A)$ of the submodule generated by the compact subset of $\LL^n(A)$ that consists of all series in $K_{\lambda}$ whose coefficient take only two values~$0$ and~$1$.

We claim that any compact subset~$C$ of~$\LL^n(A)$ is contained in~$K_{\lambda}$ for some $\lambda\in\Lambda_n$. Indeed, for any $\mu\in\Lambda_n$, the image of $C$ in the discrete \mbox{$A$-module~${\LL^n(A)/U_{\mu}}$} is a finite set whose elements are images of finite sums of monomials. Therefore condition~$(ii)$ of Lemma~\ref{lemma:intersect} holds for the subset of $\z^n$ which is the union of supports of all elements in~$C$. Thus Lemma~\ref{lemma:intersect} implies that $C$ is contained in $K_{\lambda}$ for some $\lambda\in\Lambda_n$.

Further, the map $\tau$ gives an isomorphism between the set $U_{\lambda}$ and the annihilator of the set $K_{\lambda}$. This proves that $\tau$ is a homeomorphism.
\end{proof}

Since we have an isomorphism ${\widetilde{\Omega}^n_{\LL^n(A)}\simeq \LL^n(A) dt_1\wedge\ldots\wedge dt_n}$, Proposition~\ref{prop:seldual} with ${k=(-1,\ldots,-1)}$ implies that the pairing
\begin{equation}\label{eq:pair}
\LL^n(A)\times\widetilde{\Omega}^n_{\LL^n(A)}\longrightarrow A\,,\qquad (f,\omega)\longmapsto \res(f\omega)\,,
\end{equation}
defines isomorphisms of topological $A$-modules
\begin{equation}\label{eq:dual}
{\LL^n(A)\stackrel{\sim}\longrightarrow\big(\widetilde{\Omega}^n_{\LL^n(A)}\big)^{\vee}}\,,
\qquad{\widetilde{\Omega}^n_{\LL^n(A)}\stackrel{\sim}\longrightarrow\LL^n(A)^{\vee}}\,.
\end{equation}
These isomorphisms are more useful than the isomorphisms from Proposition~\ref{prop:seldual}, because the isomorphisms~\ref{eq:dual} behave nicely under continuous endomorphisms of the \mbox{$A$-algebra~$\LL^n(A)$} as Proposition~\ref{lem:leftinverse} below claims.

\medskip

For any ${\phi\in\Hom^{\rm c,alg}_A\big(\LL^n(A),\LL^n(A)\big)}$, let ${\phi^{\vee}\in\Hom^{\rm c}_A\big(\LL^n(A),\LL^n(A)\big)}$ (see Definition~\ref{def:hom}) be the adjoint map to the continuous $A$-linear map ${\phi\colon\widetilde{\Omega}^n_{\LL^n(A)}\to \widetilde{\Omega}^n_{\LL^n(A)}}$ with respect to the pairing~\eqref{eq:pair} (see also the first  isomorphism from formula~\eqref{eq:dual}). Equivalently, for all $f\in\LL^n(A)$ and $\omega\in\widetilde{\Omega}^n_{\LL^n(A)}$, there is an equality
$$
\res\big(\phi^{\vee}(f)\,\omega\big)=\res\big(f\phi(\omega)\big)\,.
$$

\begin{rmk}\label{rmk:funvtadj}
One easily checks that the assignment $\phi\mapsto\phi^{\vee}$ is functorial with respect to~$A$, that is, given a homomorphism of rings $A\to B$, the following diagram is commutative:
$$
\begin{CD}
\Hom^{\rm c,alg}_A\big(\LL^n(A),\LL^n(A)\big) @>{\vee}>>  \Hom^{\rm c}_A\big(\LL^n(A),\LL^n(A)\big)  \\
 @VV V @VVV \\
\Hom^{\rm c,alg}_B\big(\LL^n(B),\LL^n(B)\big) @>{\vee}>> \Hom^{\rm c}_B\big(\LL^n(B),\LL^n(B)\big)\,,
\end{CD}
$$
where the horizontal maps are given by $\phi\mapsto \phi^{\vee}$
and the vertical maps are obtained by taking the extension of a continuous $A$-linear endomorphism of the \mbox{$A$-module}~$\LL^n(A)$
to a continuous $B$-linear endomorphism of the \mbox{$B$-module}~$\LL^n(B)$. This extension exists and it is unique by Proposition~\ref{prop:funccont} and Lemma~\ref{cor:conv}(ii)  (see also Corollary~\ref{cor:funccontrg} and Lemma~\ref{lem:mult}).
\end{rmk}

\begin{prop}\label{lem:leftinverse}
For any $\phi\in\Hom^{\rm c,alg}_A\big(\LL^n(A),\LL^n(A)\big)$, there is an equality (see Definition~\ref{def:ups} for $\Upsilon(\phi)$)
$$
\phi^{\vee}\circ\phi=\det\big(\Upsilon(\phi)\big)\,{\rm id}
$$
in $\Hom_A^{\rm c}\big(\LL^n(A),\LL^n(A)\big)$. In particular, if $\det\big(\Upsilon(\phi)\big)=1$, then $\phi^{\vee}\circ\phi={\rm id}$.
\end{prop}
\begin{proof}
By definition of $\phi^{\vee}$ and  Corollary~\ref{cor:res}, for all $f\in\LL^n(A)$ and $\omega\in\widetilde{\Omega}^n_{\LL^n(A)}$, we have the equalities
$$
\res\big((\phi^{\vee}\circ\phi)(f)\,\omega\big)=\res\big(\phi(f)\phi(\omega)\big)=\det\big(\Upsilon(\phi)\big)\res(f\omega)\,.
$$
Thus, the statement follows from the first of the isomorphisms~\eqref{eq:dual}.
\end{proof}

\begin{rmk}\label{rmk:expladj}
Here is an explicit formula for the adjoint map. Given an element ${\phi\in\Hom^{\rm c,alg}_A\big(\LL^n(A),\LL^n(A)\big)}$, put $\varphi_i:=\phi(t_i)$, where $1\leqslant i\leqslant n$, and define the Jacobian $J(\varphi)$ of the collection $\varphi=(\varphi_1,\ldots,\varphi_n)$ as the determinant of the matrix
$$
\left(\frac{\partial\varphi_i}{\partial t_j}\right)\in\M_{n\times n}\big(\LL^n(A)\big)\,.
$$
For short, denote $(1,\ldots,1)\in\z^n$ just by $1$. In particular, for $l=(l_1,\ldots,l_n)$, we put $l-1=(l_1-1,\ldots,l_n-1)$.
Then for any $f\in\LL^n(A)$, there is an equality
$$
\phi^{\vee}(f)=\mbox{$\sum\limits_{l\in\z^n}\res\big(f\varphi^{-l-1}J(\varphi)dt_1\wedge\ldots \wedge dt_n\big)t^l$}\,.
$$
Indeed, for any $l\in\z^n$, the $l$-th coefficient of $\phi^{\vee}(f)$ is equal to
$$
\res\big(\phi^{\vee}(f)t^{-l-1}dt_1\wedge\ldots\wedge dt_n\big)=
$$
$$
=\res\big(f\varphi^{-l-1}\phi(dt_1\wedge\ldots\wedge dt_n)\big)=\res\big(f\varphi^{-l-1}J(\varphi)dt_1\wedge\ldots \wedge dt_n\big)\,.
$$
\end{rmk}

\begin{rmk}
In general, the adjoint map $\phi^{\vee}$ is not necessarily a ring homomorphism. For example, when $n=1$ and $\phi(t)=\varphi=t^2$, Remark~\ref{rmk:expladj} implies that $\phi^{\vee}(t)=0$ and $\phi^{\vee}(t^2)=2t$. However, if $\det\big(\Upsilon(\phi)\big)=1$, then we obtain a posteriori from Theorem~\ref{theor:inv} below that $\phi^{\vee}$ is a ring homomorphism. Note that the proof of Theorem~\ref{theor:inv} is based on the theory of thick ind-cones from~\cite{GOMS}. We do not know how to deduce directly from the definition of the adjoint map that $\phi^{\vee}$ is a ring homomorphism provided that $\det\big(\Upsilon(\phi)\big)=1$.
\end{rmk}

\medskip

Now we pass to invertibility of continuous endomorphisms of the $A$-algebra $\LL^n(A)$. First we prove the following auxiliary statement.

\begin{lemma}\label{lemma:inv}
An endomorphism $\phi\in\Hom^{\rm c,alg}_A\big(\LL^n(A),\LL^n(A)\big)$ with $\det\big(\Upsilon(\phi)\big)=1$ is invertible if and only if there is an equality $\phi\circ\phi^{\vee}={\rm id}$ in $\Hom^{\rm c}_A\big(\LL^n(A),\LL^n(A)\big)$.
\end{lemma}
\begin{proof}
Suppose that $\phi$ is invertible. Then by Proposition~\ref{lem:leftinverse}, we have that $\phi^{\vee}=\phi^{-1}$, whence $\phi\circ\phi^{\vee}={\rm id}$.

Now suppose that $\phi\circ\phi^{\vee}={\rm id}$. Then Proposition~\ref{lem:leftinverse} implies that $\phi$ is a bijection from~$\LL^n(A)$ to itself, whence $\phi$ is invertible in~$\Hom^{\rm c,alg}_A\big(\LL^n(A),\LL^n(A)\big)$.
\end{proof}

Notice that Theorem~\ref{theor:inv} below claims that, in fact, the equivalent conditions of Lemma~\ref{lemma:inv} hold for any $\phi\in\Hom^{\rm c,alg}_A\big(\LL^n(A),\LL^n(A)\big)$ with $\det\big(\Upsilon(\phi)\big)=1$.

\begin{lemma}\label{lemma:auxil}
Any endomorphism $\phi\in\Ker(\Upsilon)_n^{\rm nil}(A)$ is invertible.
\end{lemma}
\begin{proof}
By the definition of $\Ker(\Upsilon)^{\rm nil}_n$, for any $i$, $1\leqslant i\leqslant n$, we have $\phi(t_i)=t_i+h_i$, where~$h_i$ is a Laurent polynomial with coefficients being nilpotent elements of~$A$. Let $\a\subset A$ be the ideal generated by all coefficients of $h_i$, $1\leqslant i\leqslant n$. Define also the ideal $I:=\a\,\LL^n(A)$ in~$\LL^n(A)$.

Clearly, the image of $\phi$ under the natural map
$$
\Hom^{\rm c,alg}_A\big(\LL^n(A),\LL^n(A)\big)\longrightarrow \Hom^{\rm c,alg}_{A/\a}\big(\LL^n(A/\a),\LL^n(A/\a)\big)
$$
coincides with the identity. Since $\a$ is finitely generated, we have an isomorphism $\LL^n(A)/I\simeq \LL^n(A/\a)$. Therefore we have that $\phi={\rm id}+h$, where $h\colon\LL^n(A)\to\LL^n(A)$ is a continuous $A$-linear map such that $h\big(\LL^n(A)\big)\subset I$.

Moreover, we claim that $h(I^k)\subset I^{k+1}$ for any $k\geqslant 0$. Indeed, we have $I^k=\a^k\,\LL^n(A)$ and for all $a\in \a^k$ and $f\in\LL^n(A)$ we have
$h(af)=ah(f)\in I^{k+1}$, because $h(f)\in I$.

 We have that $I$ is a nilpotent ideal, because  $I$ is finitely generated by nilpotent elements. Therefore  we see that the map $h$ is nilpotent. Hence the map $\phi={\rm id}+h$ is invertible.
\end{proof}

\medskip

Here is the main result of this section.

\begin{theor}\label{theor:inv}
A continuous endomorphism of the $A$-algebra $\phi\colon\LL^n(A)\to\LL^n(A)$ is invertible if and only if the matrix ${\Upsilon(\phi)\in\M^+_{n\times n}\big(\uz(A)\big)}$ is invertible, that is, the upper-triangular matrix with (locally constant on $\Spec(A)$) integral entries
$$
\Upsilon(\phi)=\big(\nu(\phi(t_1)),\ldots,\nu(\phi(t_n))\big)
$$
has units on the diagonal. Moreover, if $\phi$ is invertible, then we have the equality
$$
\phi^{-1}=\phi^{\vee}\,.
$$
\end{theor}
\begin{proof}
By Corollary~\ref{cor:hommon} and Theorem~\ref{prop:contchange}, one implication is clear. Now suppose that the matrix $\Upsilon(\phi)$ is invertible and let us show the invertibility of~$\phi$. By Corollary~\ref{cor:section}, it is enough to consider the case when $\phi\in\Ker(\Upsilon)_n(A)$, that is, the case when $\Upsilon(\phi)$ is the identity matrix, which we assume from now on.

By Lemma~\ref{lemma:inv}, we need to show that that for any $\phi\in\Ker(\Upsilon)_n(A)$, there is an equality $\phi\circ\phi^{\vee}={\rm id}$ in~$\Hom^{\rm c}_A\big(\LL^n(A),\LL^n(A)\big)$. In other words, we need to show that for all $\phi\in\Ker(\Upsilon)_n(A)$ and $f\in\LL^n(A)$, there is an equality
\begin{equation}\label{eq:main}
(\phi\circ\phi^{\vee})(f)-f=0
\end{equation}
in $\LL^n(A)$. It follows from Remark~\ref{rmk:funvtadj} that for each $l\in\z^n$, the $l$-th coefficient of the iterated Laurent series obtained in the left hand side of~\eqref{eq:main} is given by a regular function $\xi_l\in\OO\big(\Ker(\Upsilon)_n\times L^n\ga\big)$. Therefore equality~\eqref{eq:main} is equivalent to countably many equalities $\xi_l=0$, where $l\in\z^n$. Hence by Corollary~\ref{cor:repr}$(ii)$, we may assume that $\phi\in\Ker(\Upsilon)^{\rm nil}_n(A)$.

Finally, by Lemma~\ref{lemma:auxil}, any $\phi\in\Ker(\Upsilon)_n^{\rm nil}(A)$ is invertible, whence again by Lemma~\ref{lemma:inv}, we have the equality $\phi\circ\phi^{\vee}={\rm id}$. This finishes the proof.
\end{proof}

\medskip

Note that since $\phi^{-1}=\phi^{\vee}$, we obtain the explicit formula for the inverse automorphism by Remark~\ref{rmk:expladj}. In particular, for all $\varphi=(\varphi_1,\ldots,\varphi_n)\in\Hb_{n,n}(A)$ and $f\in\LL^n(A)$, there is an equality
\begin{equation}  \label{eq:identity}
f=\mbox{$\sum\limits_{l\in\z^n}\res\big(f\varphi^{-l-1}J(\varphi)dt_1\wedge\ldots \wedge dt_n\big)\varphi^l$}\,,
\end{equation}
which corresponds to the equality ${f=(\phi\circ\phi^{-1})(f)=(\phi\circ\phi^{\vee})(f)}$.
It is unclear how to prove equality~\eqref{eq:identity} directly, without using the theory of thick ind-cones developed in~\cite{GOMS}. Note that this would give a different explicit proof of Theorem~\ref{theor:inv}.

\medskip

\begin{rmk}
\hspace{0cm}
\begin{itemize}
\item[(i)]
Let $n=1$ and let ${\varphi=\sum\limits_{l\in\z}a_lt^l\in\LL(A)}$ be a Laurent series such that ${\nu(\varphi)=1}$. One checks easily that the derivative ${\partial\varphi/\partial t\in\LL(A)}$ is invertible. Applying formula~\eqref{eq:identity} with ${f=\varphi(\partial\varphi/\partial t)^{-1}t^{-1}}$, we obtain the equality
\begin{equation}\label{eq:rmkViktor}
\mbox{$\sum\limits_{l\in\z}[\varphi^{-l}]_0\,\varphi^l=\varphi\,(\partial\varphi/\partial t)^{-1}\,t^{-1}$}\,.
\end{equation}
\item[(ii)]
Suppose that $A=k$ is a field of zero characteristic. In this case, it is well-known and is easy to show that a Laurent series $\varphi\in k((t))$ with $\nu(\varphi)=1$ defines an automorphism of~$k((t))$ (cf. Theorem~\ref{theor:inv}) and that the residue map is invariant under such automorphism, see, e.g.~\cite{S} (cf. Proposition~\ref{prop:invres} and Remark~\ref{rmk:invres}). By Lemma~\ref{lemma:inv}, this proves formulas~\eqref{eq:identity} and~\eqref{eq:rmkViktor} in this case without using the theory of thick-ind cones involved in the proofs of Theorem~\ref{theor:inv} and Proposition~\ref{prop:invres}. Moreover, positive powers of~$\varphi$ have a zero constant term, whence formula~\eqref{eq:rmkViktor} becomes
$$
\mbox{$1+\sum\limits_{l\geqslant 1}[\varphi^{-l}]_0\,\varphi^l=\varphi\,(\partial\varphi/\partial t)^{-1}\,t^{-1}$}\,.
$$
Also, $\varphi=(\partial\varphi/\partial t)t$ only for $\varphi=ct$, where $c\in k^*$. Thus after replacing $\varphi$ by its inverse, we obtain the following statement.

Let $F\in\LL(k)$ be a Laurent series such that $\nu(F)=-1$ and $F\ne ct^{-1}$, where $c\in k^*$. Then the generating series of constant terms of powers of $F$, namely, the series ${\sum\limits_{l\geqslant 1}[F^l]_0\,t^l}$, is not equal to zero. Note that when $F$ is also a Laurent polynomial, this is a particular case of~\cite[Theor.\,2]{DK}.
\end{itemize}
\end{rmk}

\bigskip

Sergey Gorchinskiy

Steklov Mathematical Institute of Russian Academy of Sciences, ul. Gubkina 8, Moscow, 119991 Russia

{\em E-mail address}: gorchins@mi.ras.ru

\bigskip

Denis Osipov

Steklov Mathematical Institute of Russian Academy of Sciences, ul. Gubkina 8, Moscow, 119991 Russia

National University of Science and Technology MISIS, Leninskii pr. 4, Moscow, 119049 Russia

{\em E-mail address}: d\_osipov@mi.ras.ru

\end{document}